\documentclass{amsart}
\usepackage{amssymb, enumitem, color, tikz, hyperref, cleveref}
\newtheorem{theorem}{Theorem}
\numberwithin{theorem}{section}
\newtheorem{proposition}[theorem]{Proposition}
\newtheorem{lemma}[theorem]{Lemma}
\newtheorem{corollary}[theorem]{Corollary}

\theoremstyle{definition}
\newtheorem{definition}[theorem]{Definition}
\newtheorem{example}[theorem]{Example}
\newtheorem{remark}[theorem]{Remark}

\crefname{question}{question}{questions}

\DeclareMathOperator{\arrank}{ar-rk}
\DeclareMathOperator{\NN}{\mathbb{N}}

\DeclareMathOperator{\RR}{\mathbb{R}}
\DeclareMathOperator{\CC}{\mathbb{C}}
\DeclareMathOperator{\PP}{\mathbb{P}}
\DeclareMathOperator{\K}{\mathcal{K}}
\DeclareMathOperator{\rk}{rk}
\DeclareMathOperator{\rank}{rank}
\DeclareMathOperator{\rrank}{\RR-rk}
\DeclareMathOperator{\crank}{\CC-rk}
\DeclareMathOperator{\rbrank}{\RR-b.rk}
\DeclareMathOperator{\cbrank}{\CC-b.rk}
\DeclareMathOperator{\codim}{codim}
\DeclareMathOperator{\Sym}{Sym}
\DeclareMathOperator{\diag}{diag}
\DeclareMathOperator{\Span}{span}

\renewcommand{\Re}{\operatorname{Re}}
\renewcommand{\Im}{\operatorname{Im}}

\newcommand{\bc}{\textbf{c}}

\usepackage{fullpage}

\newcommand*{\DashedArrow}[1][]{\mathbin{\tikz [baseline=-0.25ex,-latex, dashed,#1] \draw [#1] (0pt,0.5ex) -- (1.3em,0.5ex);}}%

\title{Sums of Squares, Hankel index and Almost Real Rank}

\author{Grigoriy Blekherman, Justin Chen, Jaewoo Jung}

\address{School of Mathematics, Georgia Institute of Technology,
Atlanta, Georgia}
\email{\{greg, justin.chen\}@math.gatech.edu}
\email{jaewoojung@gatech.edu}

\thanks{Grigoriy Blekherman and Jaewoo Jung were partially supported by NSF grant DMS-1901950.}


\begin{document}

\maketitle

\begin{abstract}
The Hankel index of a real variety $X$ is an invariant that quantifies the difference between nonnegative quadrics and sums of squares on $X$.
In \cite{MR3633773} the authors proved an intriguing bound on the Hankel index in terms of the Green-Lazarsfeld index, which measures the ``linearity" of the minimal free resolution of the ideal of $X$.
In all previously known cases this bound was tight.
We provide the first class of examples where
the bound is not tight; in fact the difference between Hankel index and Green-Lazarsfeld index can be arbitrarily large.
Our examples are outer projections of rational normal curves, where we identify the center of projection with a binary form $F$.
The Green-Lazarsfeld index of the projected curve is given by the complex Waring border rank of $F$ \cite{MR2299577}.
We show that the Hankel index is given by the \emph{almost real} rank of $F$, which is a new notion that comes from decomposing $F$ as a sum of powers of almost real forms.
Finally, we determine the range of possible and typical almost real ranks for binary forms.
\end{abstract}

\section{Introduction}
The relationship between nonnegative polynomials and sums of squares is a fundamental topic in real algebraic geometry.
This subject has received renewed attention in the last twenty years due to its connection with polynomial optimization and many applications \cite{MR3075433}.
In a foundational paper Hilbert described all the cases in terms of degree and number of variables where any globally nonnegative polynomial can be written as a sum of squares of polynomials \cite{MR1510517}.
A modern approach to this question is to study nonnegative polynomials and sums of squares on a real projective variety $X\subseteq \PP^n_{\RR}$.
This allows one to restrict to quadrics, since degree $2d$ forms on $X$ are quadrics on the $d$-th Veronese embedding of $X$.
The two main objects of interest are:
\begin{align*}
    P_X &:= \left\{ f \in R(X)_2 \mid f(x) \ge 0 \,\, \text{for all} \,\, x \in X(\RR) \right\}, \\
    \Sigma_X &:= \left\{ f \in R(X)_2 \mid \,\, \text{there exist} \,\,  l_1, \ldots, l_m \in R(X)_1, f = \sum_{i=1}^m l_i^2 \right\}
\end{align*}
In fact $\Sigma_X \subseteq P_X$ are convex cones in the vector space $R(X)_2$ of all quadrics on $X$, which facilitates their study via convex geometry (cf. \cite{MR3075433}).
For instance, as an extension of Hilbert's result, \cite[Theorem 1.1]{MR3486176} showed that $\Sigma_X = P_X$ if and only if $X$ is a variety of minimal degree, i.e. $\deg X = 1 + \codim X$.
However, the structure of these cones is still not well understood in general.

It is sometimes more convenient to work with the dual cones $P_X^\star \subseteq \Sigma_X^\star$.
The cone $\Sigma_X^\star$ is a \emph{spectrahedron}, i.e. a slice of the cone of positive semidefinite (PSD) matrices with a linear subspace.
We call $\Sigma_X^\star$ the \emph{Hankel spectrahedron} of $X$.
By identifying a point $\ell$ in $\Sigma_X^\star$ with a PSD matrix we can talk about the rank of $\ell$.
Rank one extreme rays of $\Sigma_X^*$ are precisely the extreme rays of $P_X^\star$.
Therefore, if $P_X^*\subsetneq \Sigma_X^*$
we can quantitatively measure the difference between these cones by analyzing the ranks of extreme rays of $\Sigma_X^*$ that are greater than one.
This motivates the following key definition:

\begin{definition}(cf. \cite[Definition 1]{MR3633773})
The \emph{Hankel index} of $X$, denoted $\eta(X)$, is defined to be the minimal rank of a(n extreme) ray $\ell \in \Sigma_X^\star \setminus P_X^\star$, or $\infty$ if $\Sigma_X^\star = P_X^\star$.
\end{definition}

The Hankel index is a subtle invariant which is often quite hard to compute.
A surprising connection between the Hankel index and homological properties of the minimal free resolution of the ideal of $X$ was found in \cite[Theorem 4 and Theorem 6]{MR3633773}: namely, there is a lower bound $\eta(X) \ge \alpha(X) + 1$, where $\alpha(X)$ is the Green-Lazarsfeld index of $X$ (here $X$ need not be irreducible).
Recall that the Green-Lazarsfeld index of $X$ is defined as follows: $\alpha(X)=0$ if the ideal of $X$ is not generated by quadrics; otherwise it is equal to one plus the number of steps that the minimal free resolution of the coordinate ring of $X$ is linear, i.e. has only linear syzygies.
In all cases where the Hankel index was known, this bound was tight.
These cases include varieties of minimal degree, arithmetically Cohen-Macaulay (ACM) varieties of almost minimal degree, varieties defined by quadratic squarefree monomial ideals, some general canonical curves, and Veronese embeddings of $\PP^2$ (see \cite[Theorem 28]{MR3633773} and \cite{MR3550352}).

We present the first examples where the difference between Hankel index and Green-Lazarsfeld index is larger than one.
The Hankel index of $X$ is a semialgebraic invariant, while the Green-Lazarsfeld index is an algebraic invariant which makes no distinction between the real and complex points of $X$.
Nevertheless separating these two invariants is challenging.
To accomplish this we consider non-ACM curves of almost minimal degree.
This class of curves is well-studied: such curves admit a description as an outer projection of a rational normal curve, and are thus determined by a single point, namely the projection center \cite{MR2274517}.
Since we are working with the rational normal curve, we can identify the projection center $p$ with a binary form $F(p)$ (cf. \Cref{ssec:pointsToForms}).
In this case, both Green-Lazarsfeld and Hankel index are intimately connected to another classical notion: \emph{Waring decomposition} of $F(p)$, i.e. shortest decomposition of $F(p)$ as a sum of powers of linear forms.
In \cite[Theorem 1.1(2)]{MR2299577} it was shown that for such curves, the Green-Lazarsfeld index equals the complex Waring border rank of $F(p)$ minus 3: $\alpha(X) = \cbrank(F(p)) - 3$.
The Hankel index of $X$ is determined by the shortest decomposition of $F(p)$ into as a sum of powers of \emph{almost real forms}, (cf. \Cref{ssec:almostRealDef} for precise definitions), which we call the \emph{almost real rank} of $F(p)$.

\begin{theorem} \label{thm:mainThm}
Let $X = \pi_p(C_d)$ be a projection of a rational normal curve $C_d$ of degree $d$ away from a point $p \in \PP^d \setminus C_d^3$, with corresponding binary form $F(p) \in \RR[x,y]_d$.
Then the Hankel index of $X$ is given by 
\[
\eta(X) = \arrank(F(p)) - 2.
\]
\end{theorem}
\noindent
This theorem elucidates the semialgebraic nature of the Hankel index, and demonstrates two ways in which it differs from the Green-Lazarsfeld index: the difference between rank and border rank, and the difference between almost real decompositions and complex decompositions.

We note an interesting technical detail of the proof of Theorem~\ref{thm:mainThm}.
To prove an upper bound on Hankel index we need a construction of rays in $\Sigma_X^*\setminus P_X^*$, and for this we use point evaluations at points of $X$ in special position.
Such constructions using Cayley-Bacharach relations were used in \cite{blekherman2012nonnegative} and more generally in \cite{MR3486176} (the idea goes all the way to Hilbert's original proof).
Until now these construction only used reduced points of $X$, but in this paper we use non-reduced $0$-dimensional subschemes of $X$.
The use of such non-reduced configurations is necessary, and cannot be replicated by reduced points.

Real and complex Waring decomposition of binary forms is a classical subject dating back to Sylvester \cite{sylvester1904remarkable,sylvester1865syllabus}.
The notion of almost real rank is new and we prove some basic results about almost real rank of binary forms.
We show that the maximal almost real rank for degree $d$ forms is $d-1$, and classify all forms of maximal almost real rank (Theorem \ref{thm:maxARrankChar}).
We also show that the range of typical almost real ranks $r$ is precisely $\lfloor\frac{d+2}{2}\rfloor \le r \le d-2$ (\Cref{thm:typr}).

We outline the paper as follows: Sections \ref{sec:apolarRanks} and \ref{sec:setup2} introduce necessary background and setup, including the notion of almost real rank. Section \ref{sec:monomial} consists of a small explicit example illustrating construction techniques presented in Section \ref{sec:construction}. Sections \ref{sec:construction} and \ref{sec:lowerBound} constitute the proof of \Cref{thm:mainThm} (covering the inequalities ``$\le$'' and ``$\ge$'' respectively).
We conclude in \Cref{sec:almostRealRank} with an investigation of almost real rank for binary forms.

\section{Apolarity and Ranks} \label{sec:apolarRanks}

We begin with a brief review of apolarity and the apolar inner product, which is our preferred method of explicitly identifying primal and dual spaces.

\begin{definition} \label{def:apolar}
Let $k$ be a field of characteristic 0, and $R = k[x_1, \ldots, x_n]$ a polynomial ring over $k$.
Consider the ``differential'' pairing on $R$ defined by 
\begin{align} \label{eq:innerProduct}
    \langle f, g \rangle := \partial(f) \bullet g
\end{align}
where $\partial(f)$ is the differential operator obtained from $f$ by replacing each variable $x_i$ with $\frac{\partial}{\partial x_i}$, and $\bullet$ denotes the action of differential operators on polynomials.
For a given degree $d$, the pairing $\langle \cdot, \cdot \rangle$ restricts to an inner product on $R_d$, the $k$-vector space of forms of degree $d$.
% Indeed, the Gram matrix of this inner product in the monomial basis \{ x^\alpha \mid |\alpha| = d \} is given by \langle x^\alpha, x^\beta \rangle = \alpha! if \alpha = \beta, and 0 otherwise.
For $F \in R$, the \emph{apolar ideal} of $F$ is defined as the orthogonal complement of $F$ with respect to the pairing (\ref{eq:innerProduct}), i.e.
\[
(F)^\perp := \{ f \in R \mid \langle f, F \rangle = 0 \}.
\]
If $F \in R_d$ is homogeneous, then $(F)^\perp$ is a homogeneous ideal.
\end{definition}

\begin{remark} \label{rem:apolar1}
For any form $F$, the apolar ideal $(F)^\perp$ is an Artinian Gorenstein graded ideal.
Conversely, every Artinian Gorenstein graded ideal $I$ is of the form $(F)^\perp$, where $F$ generates the socle of $R/I$.
\end{remark}

We now specialize to the case of binary forms, i.e. forms in 2 variables $x, y$.
Let $F \in k[x,y]_d$ be a binary form.
Then $(F)^\perp$ is Gorenstein of codimension $2$, hence is a complete intersection. 
As this fact will be used repeatedly in the sequel, we introduce some notation for the generators of this complete intersection:
% say $(F)^\perp =: (g, h)$ with $g, h$ homogeneous and $\deg g + \deg h = \deg F + 2$ (throughout, we always take $\deg g \le \deg h$).

\begin{definition} \label{def:apolarGens}
For $F \in k[x,y]_d$, let $F_\perp, F^\circ \in k[x,y]$ denote forms that satisfy
\begin{align*}
    (F)^\perp = (F_\perp, F^\circ)
\end{align*}
with $\deg F_\perp \le \deg F^\circ$.
If $d_1 := \deg F_\perp$ and $d_2 := \deg F^\circ$, we say that the apolar ideal $(F)^\perp$ is \emph{of type} $(d_1, d_2)$.
One always has the relation 
\begin{align} \label{eq:apolarDegRelation}
    d_1 + d_2 = \deg F + 2.
\end{align}
Note that if $d_1 < d_2$, then $F_\perp$ is uniquely defined by $F$ (up to nonzero scale), while $F^\circ$ is unique modulo the principal ideal $(F_\perp)$.

For example, if $l = ax + by \in k[x,y]_1$ is a binary linear form, then $(l)^\perp$ is of type $(1,2)$, with $l_\perp = bx - ay$, and $l^\circ$ is (the) quadric not in $(l_\perp)$.
% If $l$ is a binary linear form, $l = ax + by \in k[x,y]_1$, define 
% \[
% l^\perp := bx - ay
% \]
% as the \emph{orthogonal} linear form to $l$ (note that $l^\perp$ is the (unique up to scale) linear generator of the apolar ideal $(l)^\perp$).
\end{definition}

% \begin{lemma}[Apolarity lemma] \label{lem:apolarity}
% Let $F \in k[x,y]_d$. For a given set $\{ l_1, \ldots, l_r \} \subseteq k[x, y]_1$ of distinct linear forms, one has $\prod_{i=1}^r l_i \in (F)^\perp$ iff there exist $c_i \in k$ $(1 \le i \le r)$ such that 
% \[
% F = \sum_{i=1}^r c_i (l_i^\perp)^d.
% \]
% \end{lemma}

We are now ready to state the apolarity lemma for binary forms, which characterizes membership in the apolar ideal:

\begin{lemma}[Generalized Apolarity Lemma] cf.\cite[Lemma 1.31]{MR1735271} \label{lem:genApolarity}
Let $F \in k[x, y]_d$. For a given set $\{ l_1, \ldots, l_r \} \subseteq k[x, y]_1$ of linear forms and $d_1, \ldots, d_r \in \NN$ with $\sum_{i=1}^r d_i \le d$, one has $\prod_{i=1}^r l_i^{d_i} \in (F)^\perp$ if and only if there exist $c_{ij} \in k$ $(1 \le i \le r$, $0 \le j \le d_i-1)$ such that
\[
F = \sum_{i=1}^r \sum_{j=0}^{d_i-1} c_{ij} (l_i)^j(l_i)_\perp^{d-j}.
\]
\end{lemma}

The case $d_1 = \ldots = d_r = 1$ is classically referred to as the \emph{apolarity lemma}, and characterizes squarefree forms in the apolar ideal via a Waring decomposition of $F$, as a sum of $d^\text{th}$ powers of linear forms.

Another useful criterion for determining membership in the apolar ideal is:

\begin{lemma} \label{lem:apolarMembershipSameDegree}
Let $F \in k[x,y]_d$, and $G \in k[x,y]_n$ for some $n \le d$.
Then $G \in (F)^\perp$ if and only if $(G)_d \subseteq (F)^\perp$.
\end{lemma}

\begin{proof}
If $G \in (F)^\perp$, then certainly $(G)_d \subseteq (F)^\perp$, since $(F)^\perp$ is an ideal.
Conversely, suppose $G \not \in (F)^\perp$, and set $H := \langle G, F \rangle \in k[x,y]_{d-n} \ne 0$.
Since $\langle \cdot, \cdot \rangle$ is a perfect pairing on $k[x,y]_{d-n}$, there exists $0 \ne K \in k[x,y]_{d-n}$ with $\langle K, H \rangle \ne 0$.
Then $0 \ne \langle K, \partial(G) \bullet F \rangle = \partial(K) \bullet (\partial(G) \bullet F) = \partial(KG) \bullet F$, so $KG \in (G)_d \setminus (F)^\perp$.
\end{proof}

\subsection{Ranks of forms} \label{ssec:ranks}
Classically, it is an important problem to decompose a given form as a linear combination of powers of linear forms.
Such decompositions lead various notions of rank of a form, which are sensitive to the underlying field of scalars.

\begin{definition} \label{def:ranks}
Let $F \in \RR[x,y]_d$.
The \emph{real} (resp. \emph{complex}) \emph{rank} of $F$ is the minimal number of real (resp. complex) linear forms $l_1, \ldots, l_r$ such that $F$ is an $\RR$-linear (resp. $\CC$-linear) combination of $l_1^d, \ldots, l_r^d$.
% The \emph{complex rank} of $F$ is the minimal number of complex linear forms $l_1, \ldots, l_r$ such that $F$ is an $\CC$-linear combination of $l_1^d, \ldots, l_r^d$.
The \emph{real} (resp. \emph{complex}) \emph{border rank} of $F$ is the minimal number $r$ such that $F$ is a limit of forms of real (resp. complex) rank $r$.
\end{definition}

\begin{remark} \label{rem:ranksViaApolar}
Via apolarity, we can reinterpret the various ranks in \Cref{def:ranks}.
Indeed, it follows from \Cref{lem:genApolarity} that for any $F \in \RR[x,y]_d$,
    \[
    \rrank(F) = \min \left\{ r \biggm| 
    \begin{array}{cc}
        \exists g \in (F)^\perp_r \text{ with } r \text{ simple}\\
        \text{ linear factors over } \RR
    \end{array} \right\}
    \]
    \[
    \crank(F) = \min \left\{ r \biggm| 
    \begin{array}{cc}
        \exists g \in (F)^\perp_r \text{ with } r \text{ simple}\\
        \text{ linear factors over } \CC
    \end{array} \right\}
    \]
    \[
    \rbrank(F) = \min \left\{ r \biggm| 
    \begin{array}{cc}
        \exists g \in (F)^\perp_r \text{ which factors }\\
        \text{ completely over } \RR
    \end{array} \right\}
    \]
    \[
    \cbrank(F) = \min \{ r \mid (F)^\perp_r \ne 0 \} = \deg(F_\perp)
    \]
    
Note that any complex rank is at most the corresponding real rank, and any border rank is at most the corresponding non-border rank.
Moreover, if $(F)^\perp$ is of type $(d_1, d_2)$, then $\crank(F) = d_1$ if and only if $F_\perp$ has distinct factors over $\CC$, and equals $d_2$ otherwise (since $F_\perp, F^\circ$ form a complete intersection, thus have no common factors).
\end{remark}

\subsection{Almost reality} \label{ssec:almostRealDef}
We now introduce a central notion for this article, which is that of a binary form almost splitting over $\RR$, or a univariate polynomial having almost all real roots.
For technical reasons we will need to include the possibility of one pair of roots being nondistinct, so that the resulting rank is intermediate between a true rank and a border rank.

\begin{definition} \label{def:almostreal}
Let $F \in \RR[x,y]_d$.
We say that $F$ has \emph{almost real roots} if $F$ has $\ge d-2$ simple linear factors over $\RR$.
Equivalently, $F$ has a factorization over $\RR$ of the form 
\[
F = q \cdot \prod_{i=1}^{d-2} l_i, \quad \{l_1, \ldots, l_{d-2}, q\} \text{  pairwise relatively prime}
\]
where $l_i$ are linear and $q$ is quadratic.
A polynomial $F$ with almost real roots thus belongs to exactly one of 3 classes: (i) $F$ has all simple real roots, (ii) $F$ has a unique nonreal complex conjugate pair of roots, (iii) $F$ has a unique double real root (note that in cases (ii) and (iii), all other roots are real and simple).

In analogy with \Cref{rem:ranksViaApolar}, we define the \emph{almost real rank} of $F$ as
    \[
    \arrank(F) := \min \left\{ r \biggm| 
    \begin{array}{cc}
        \exists g \in (F)^\perp_r \text{ with} \\
        \text{almost real roots}
    \end{array} \right\}
    \]
\end{definition}

% \begin{remark} \label{rem:arrankComparisons}
% \begin{align*}
%     \cbrank(F) \le \arrank(F) \le \rrank(F),
% \end{align*}
% \end{remark}

\begin{remark} \label{rem:arrankGenDef}
One can generalize the definition above to arbitrary (i.e. not necessarily binary) forms.
Given a form $F \in \RR[x_1, \ldots, x_n]$, define the \emph{almost real rank} of $F$ as the minimal length of a zero-dimensional subscheme $Z \subseteq \PP^{n-1}_{\CC}$ such that $I(Z) \subseteq (F)^\perp$ and $Z$ has either (i) all reduced real points, or (ii) exactly 1 nonreal conjugate pair of points, or (iii) exactly 1 double point.
In this article though, we will only use the notion of almost real rank for binary forms.
\end{remark}

Note that for any $F \in \RR[x,y]_d$, it follows from the definitions that $\cbrank(F) \le \arrank(F) \le \rrank(F)$.
For more properties of almost real rank, see \Cref{sec:almostRealRank}.

\section{From binary forms to quadrics} \label{sec:setup2}

\subsection{Associating forms to points} \label{ssec:pointsToForms}
A crucial identification throughout this paper is that of associating points in projective space to (binary) forms, which we now explain.
Let $\nu_d : \PP^1 \to \PP^d$ be the $d$-uple embedding (or $d^\text{th}$ Veronese map).
Let $C_d := \nu_d(\PP^1) \subseteq \PP^d$ be the image, which is the standard rational normal curve of degree $d$.
Given a point $p \in \PP^d$, consider the vector space of linear forms on $\PP^d$ vanishing at $p$ (these generate the vanishing ideal of $p$).
Pulling this space back to $\PP^1$ via $\nu_d$ gives a $d$-dimensional vector space of degree $d$ binary forms, which is a hyperplane in $k[x,y]_d$ (the space of all degree $d$ binary forms).
We set $F(p)$ to be the degree $d$ binary form (unique up to nonzero scale) which is orthogonal to this hyperplane, with respect to the inner product (\ref{eq:innerProduct}).

An alternate way to compute $F(p)$ is: under the $d$-uple embedding, a point $\nu_d([a : b])$ on the rational normal curve is associated to the $d^\text{th}$-power $(ax+by)^d \in k[x,y]_d$.
Since points on the rational normal curve are in linearly general position, extending additively gives a correspondence between all points in $\PP^d$ and binary forms of degree $d$.
Explicitly, for $p \in \PP^d$, we may choose an expression of $p$ as a linear combination of $r \le d+1$ points on $C_d$, say $p = \sum_{i=1}^r c_i p_i$.
Setting $p_i =: \nu_d([a_i : b_i])$, we have
\[
F(p) = \sum_{i=1}^r c_i(a_ix + b_iy)^d \in k[x,y]_d
\]
In this way we may consider the various ranks (defined in \Cref{ssec:ranks,ssec:almostRealDef}) of a point $p \in \PP^d$, as the ranks of the associated binary form $F(p)$.

\subsection{Quadratic forms vs linear functionals on quadrics} \label{ssec:quadFormsVsFunctionals}
For an embedded nondegenerate projective variety $X \subseteq \PP^n$, there is a correspondence between quadratic forms on $X$ and linear functionals on quadrics on $X$.
Let $R = R(X) = \bigoplus_{i \ge 0} R_i$ be the homogeneous coordinate ring of $X$.
A bilinear form on $R_1$ is a bilinear map $R_1 \times R_1 \to k$, or equivalently a linear map $R_1 \otimes_k R_1 \to k$.
The bilinear form is symmetric if and only if this descends to $\Sym^2(R_1) \to k$.
Since $X$ is nondegenerate, $\dim R_1 = n+1$ (i.e. $R_1$ consists of all linear forms on $\PP^d$), so there is a natural surjection $\Sym^2(R_1) \twoheadrightarrow R_2$ with kernel $I(X)_2$, the degree $2$ part of the defining ideal of $X$.
% Thus a linear functional on quadrics $R_2 \to k$ gives a symmetric bilinear form $\Sym^2(R_1) \to k$ by pre-composing with $\Sym^2(R_1) \twoheadrightarrow R_2$, and conversely a symmetric bilinear form whose kernel contains $I_2(X)$ descends to a functional $R_2 \to k$.
This yields a bijection
\[
\left\{
    \begin{array}{cc}
        \text{ symmetric bilinear forms on } R_1\\
        \text{whose kernel contains } I(X)_2
    \end{array}
\right\}
\longleftrightarrow
\left\{ 
\text{ linear functionals on } R_2
\right\}
\]
Finally, symmetric bilinear forms on $R_1$ whose kernel contains $I(X)_2$ correspond to quadratic forms on the variety $X$.
Explicitly, given $\ell \in R(X)_2^\star$, we associate to $\ell$ a quadratic form $Q_\ell$ on $R(X)_1$ given by $Q_\ell(f) := \ell(f^2)$.

\subsection{Curves of almost minimal degree} \label{ssec:curves}
We now specialize to the main class of varieties of interest in this paper.
Since $P_X$ only depends on real points of $X$, it is natural to restrict to \emph{totally real} varieties (i.e. real varieties whose set of real points is Zariski-dense), and since $\Sigma_X$ only depends on the quadratic part of the coordinate ring of $X$, it is important to restrict to varieties defined by quadrics.
We consider smooth projective non-ACM curves of almost minimal degree.
Such curves arise as projections of the rational normal curve $C_d$ from a point (cf. \cite[Theorem 1.2]{MR2274517}).
Let $C_d^3$ denote the $3^\text{rd}$ secant variety of $C_d$, i.e. the Zariski closure of the union of all secant $2$-planes to $C_d$ in $\PP^d$, meeting $C_d$ in $3$ distinct points. 
For $p \in \PP^d \setminus C_d^3$, let $\pi_p : \PP^d \DashedArrow[->] \PP^{d-1}$ be projection with center $p$ (i.e. away from $p$).
On restriction to $C_d$, the rational map $\pi_p$ becomes a morphism, and the image $X := \pi_p(C_d) \subseteq \PP^{d-1}$ is a smooth rational curve of almost minimal degree $d = \deg X = \codim X + 2$.
Let $R(X) := \mathbb{R}[x_0, \ldots, x_{d-1}]/I(X)$ denote the real coordinate ring of $X$.
The assumption that $p \not \in C_d^3$ is equivalent to the statement that $I(X)$ is generated by quadrics, cf. \cite[Theorem 1.1(2)]{MR2299577}.
Since $X$ is projective, $R(X) = \bigoplus_{i=0}^\infty R(X)_i$ is naturally $\mathbb{Z}$-graded.
\begin{figure}[h]
    \centering
    \includegraphics[scale=0.3]{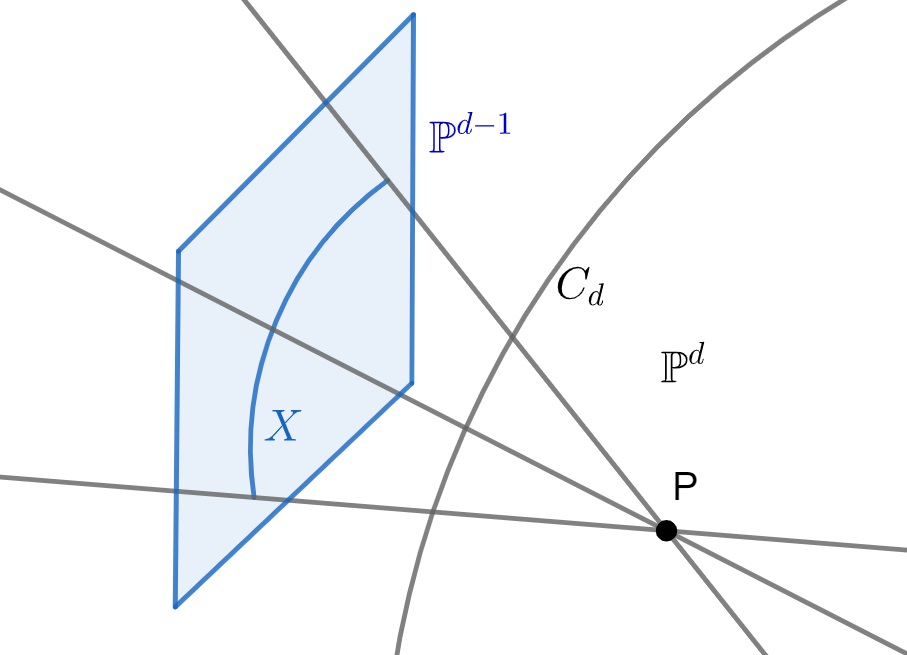}
    \caption{Projection of the rational normal curve $C_d \subseteq \PP^d$ away from a point $p$}
    \label{fig:Curve}
\end{figure}

Our main object of interest is the Hankel spectrahedron
\[
\Sigma_X^\star := \{ \ell \in R(X)_2^* \mid \ell(f^2) \ge 0, \,\, \text{for all} \,\, f \in R(X)_1 \}
\]
This is the dual cone to the sums-of-squares cone of $X$, and is contained in $R(X)_2^\star$, the space of linear functionals on quadrics on $X$.
That it is a spectrahedron can be seen from an alternate description (cf. \cite[Lemma 2.1]{blekherman2012nonnegative} and \Cref{ssec:quadFormsVsFunctionals})
\[
\Sigma_X^\star = \mathbb{S}_{+} \cap (I(X)_2)^\perp
\]
where $\mathbb{S}_{+}$ is the cone of PSD symmetric matrices (identified with nonnegative quadratic forms) on $X$, and $(I(X)_2)^\perp$ is the orthogonal complement of the degree 2 part of the ideal of $X$ (which comprises linear equations in $R(X)_2^*$).
We next spell out a series of basic, but useful, identifications.

\begin{remark} \label{rem:identifications}
(i) The surjection $\pi_p : C_d \twoheadrightarrow X$ induces an injection of coordinate rings $R(X) \hookrightarrow R(C_d)$, which is naturally graded.
In this way $R(X)_1$ is identified with a hyperplane $H \subseteq R(C_d)_1$.

(ii) Since $p \not \in C_d^3$, the quadratic part of the coordinate ring of $X$ can be identified with the quadratic part of the coordinate ring of $C_d$, i.e. $R(X)_2 = R(C_d)_2$.
Equivalently, the Hilbert function of $X$ in degree $2$ has value $2d+1$.

(iii) Via the $d$-uple embedding $\nu_d : \PP^1 \to \PP^d$, $R(C_d)_1$ can in turn be identified with $R(\PP^1)_{d} = \RR[x,y]_{d}$, the space of all degree $d$ binary forms, and similarly $R(C_d)_2 \cong \RR[x,y]_{2d}$.

(iv) The apolar inner product (\ref{eq:innerProduct}) on $\RR[x,y]_{d}$, along with (iii), gives an explicit description of the hyperplane $H$ in (i): namely $H$ is the orthogonal complement in $\RR[x,y]_d$ of the center $F(p)$ (cf. \Cref{ssec:pointsToForms}), which is also $(F(p))^\perp_d$, the degree $d$ part of the apolar ideal of $F(p)$.
Moreover, with respect to the pairing on $\RR[x,y]_{2d}$, every functional $\ell \in \RR[x,y]_{2d}^*$ can be realized as $\ell(\cdot) = \langle \cdot, L \rangle$ for some $L \in \RR[x,y]_{2d}$.

% v) For any degree $d$, we fix the inner product on $\RR[x,y]_d$ given by differentiation: $\langle f, g \rangle := \partial(f) \cdot g$, and take the induced inner product on $R(X)_2$ via the identifications above.

(v) Putting (i) -- (iv) together with \Cref{ssec:quadFormsVsFunctionals}, we may thus associate to any $\ell \in \Sigma_X^\star$ a binary form $L \in \RR[x,y]_{2d}$, as well as quadratic forms $Q_\ell$ acting on $\RR[x,y]_d \cong R(C_d)_1$ and $q_\ell$ acting on $(F(p))^\perp_d \cong R(X)_1$.
Note that $q_\ell = Q_\ell \Big|_H$ is the restriction of $Q_\ell$ to $H$: when represented as symmetric matrices, $Q_\ell$ is $(d+1) \times (d+1)$, whereas $q_\ell$ is $d \times d$.
\end{remark}

We briefly review what is known about algebraic invariants of curves of almost minimal degree.
First, for any nondegenerate variety $Y \subseteq \PP^n_{\CC}$, there is a stratification of $\PP^n$ by (higher) secant varieties of $Y$:
\[
Y \subsetneq Y^2 \subsetneq Y^3 \subsetneq \cdots \subsetneq Y^{k-1} \subsetneq Y^{k} = \PP^n
\]
This gives rise to the notion of $Y$-border rank: for $p \in \PP^n$, the $Y$-border rank of $p$ is defined as $\rk_Y(p) := \min \{ i \mid p \in Y^i \}$ (cf. \cite{MR3368091,MR2628829}).
For $Y = C_d$, it follows from \Cref{ssec:pointsToForms} and apolarity that the $C_d$-border rank of a point is exactly the complex border rank of the corresponding binary form, i.e. $\rk_{C_d}(p) = \cbrank(F(p))$.

Next, a fruitful way to study a projected curve $X = \pi_p(C_d)$ is to consider the rational normal scrolls containing $X$ as a divisor.
% \jj{We may think our case is the curve case of Theorem 1.3 of \cite{MR2331758}. We may be able to mention it with citation.}
Recall that a rational normal scroll is a variety $S(a_1, \ldots, a_m)$ which is a join of disjoint rational normal curves of degrees $a_1, \ldots, a_m$ in $\PP^{\sum_{i=1}^m (a_i+1)-1}$; the tuple $(a_1, \ldots, a_m)$ is called the \emph{type} of the scroll.
As $\dim S(a_1, \ldots, a_m) = m$ and $\deg S(a_1, \ldots, a_m) = \sum_{i=1}^m a_i$, every scroll is a variety of minimal degree, and conversely any nondegenerate variety of minimal degree is either a quadric hypersurface, the second Veronese of $\PP^2$, or a scroll (cf. \cite{MR927946}).
It was shown in \cite{MR2299577} that the Green-Lazarsfeld index of $X$ (and even the entire graded Betti table of $X$) is determined by the types of surface scrolls containing $X$, which in turn is determined by $\rk_{C_d}(p)$:

\begin{theorem}[{\cite[Theorem 1.1]{MR2299577}}] \label{thm:GLindex}
    Let $C_d \subseteq \PP^d$ be a rational normal curve of degree $d$, $\pi_p : \PP^d \DashedArrow[->] \PP^{d-1}$ the projection away from a point $p \in \PP^d \setminus C_d^2$, and $X := \pi_p(C_d) \subseteq \PP^{d-1}$.
    Then
    \begin{enumerate}
        \item $X$ is contained in a surface scroll $S(a,b)$ with $1 \le a \le b$ if and only if $a = \rk_{C_d}(p) - 2$, and
        \item The Green-Lazarsfeld index of $X$ is given by $\alpha(X) = \rk_{C_d}(p)-3$.
    \end{enumerate}
\end{theorem}

This implies that $$\cbrank(F(p))-2 = \alpha(X) + 1 \le \eta(X)$$
by \cite[Theorems 4, 6]{MR3633773}.
We will strengthen this inequality in \Cref{thm:lowerBound}.

\subsection{Kernels of rays}

\begin{definition}
Let $\K \subseteq \RR^n$ be a convex cone, and $\ell \in \K$. We say that $\ell$ spans an \emph{extreme ray} of $\K$ if whenever $\ell = \ell_1 + \ell_2$ with $\ell_1, \ell_2 \in \K$, one has $\ell_1 = \lambda_1 \ell$, $\ell_2 = \lambda_2 \ell$ for some $\lambda_1, \lambda_2 \in \RR$.
\end{definition}

If $\ell \in \K$ spans an extreme ray of $\K$, we will simply say that $\ell$ is an extreme ray of $\K$ (i.e. we do not distinguish an extreme ray from its nonzero elements).
For instance, we can say that every $\ell \in \K$ can be written as a sum of extreme rays.

\begin{proposition}[{\cite[Lemma 2.2]{blekherman2012nonnegative}}] \label{prop:extremeRaysChar}
Let $\K = \mathbb{S}_{+} \cap L$ be a spectrahedron, and $\ell \in \K$.
Then $\ell$ is an extreme ray of $\K$ if and only if $\ker \ell$ is maximal, i.e. if $\ker \ell \subseteq \ker \ell'$ for some $\ell' \in L$, then $\ell' = \lambda \ell$ for some $\lambda \in \RR$.
\end{proposition}

The simplest extreme rays in $\Sigma_X^\star$ are given by point evaluations.
For a point $p \in X$, we can pick an affine representative $\tilde{p}$ lying on the line spanned by $p$, and define a linear functional $\ell_{\tilde{p}}(q) := q(\tilde{p})$ for all $q \in R(X)_2$.
Varying the affine representative only rescales the point evaluation functional, and so by a slight abuse of terminology we will talk about point evaluations at a point $p \in X$ and use $\ell_p$ to denote any of the linear functionals obtained by using an affine representative of $p$.
Point evaluations are precisely the rank $1$ quadratic forms in $\Sigma_X^\star$: if $\ell \in \Sigma_X^\star$ has $\rank Q_\ell = 1$, then $\ell = \ell_p$ for some $p \in X$ \cite[Lemma 2.3]{MR3486176}.

Recall that if $V \subseteq R_d$ is a space of forms, then a point $p$ is called a \emph{basepoint} of $V$ if all forms in $V$ vanish at $p$. If $V$ has no basepoints, we say that $V$ is \emph{basepoint-free}.

\begin{remark} \label{rem:basePtFreeVsPtEval}
We take a moment to clarify the relationships between rays with basepoint-free kernels and sums of point evaluations. %\gb{we may want to have special notation for point evaluations. Below we sometimes implicitly assume that $\ell_i$ is a point evaluation and sometimes not.}

(i) For $\ell_i \in \Sigma_X^\star$, $\ker(\sum q_{\ell_i}) = \bigcap \ker(q_{\ell_i})$:
for $v \in R(X)_1$, one has $q_{\ell_i}(v) \ge 0$ with equality if and only if $v \in \ker(q_{\ell_i})$, as $q_{\ell_i}$ is PSD.

(ii) If $\ell_p$ is point evaluation at a point $p \in X$, then $p$ is a basepoint of $\ker(q_{\ell_p})$.

(iii) It follows from (i) and (ii) that if $\ell \in \Sigma_X^\star$ is such that $\ker(q_\ell)$ is basepoint-free, then for any decomposition of $\ell$ as a sum of extreme rays $\ell = \sum \ell_i$ of $\Sigma_X^\star$, each extreme ray $\ell_i$ has rank $> 1$, i.e. is not a point evaluation.
(In fact the converse holds as well: if $p$ is a basepoint of $\ker(q_\ell)$, then there is a decomposition of $\ell$ into extreme rays, one of which is $\ell_p$.
However, note that a sum of extreme rays of rank $> 1$ may have basepoints.)
\end{remark}

The next lemma connects kernels of quadratic forms to apolar ideals of binary forms, which is key for our main result.

\begin{lemma} \label{lem:kerQuadFormViaApolar}
Let $d \ge 1$, $L \in \RR[x,y]_{2d}$, and $Q$ the quadratic form on $\RR[x,y]_d$ associated to the functional $\langle \cdot, L \rangle$ (as in \Cref{rem:identifications}).
Then $\ker(Q) = (L)^\perp_d$.
\end{lemma}

\begin{proof}
The matrix $A$ of $Q$ is constructed with respect to a basis $B = \{b_0, \ldots, b_d\}$ of $\RR[x,y]_d$ as follows: the $(i,j)$ entry of $A$ is $\langle b_ib_j, L \rangle$.
Given $f \in \RR[x,y]_d$, one has $f \in \ker(Q) \iff Q(b_i f) = 0$ for all $0 \le i \le d \iff f \in (L)^\perp$, by \Cref{lem:apolarMembershipSameDegree}.
\end{proof}

We also note that vanishing at points on $\PP^1$ with specified multiplicities imposes independent conditions on binary forms.
% which will be useful in showing that the kernel of a constructed ray is basepoint-free.

\begin{proposition} \label{prop:vanishingIndepCond}
Let $d \ge 0$, $\{ P_1, \ldots, P_r \} \subseteq \PP^1$ and $r_1, \ldots, r_r \in \NN$ be given. 
Then the space of degree $d$ binary forms vanishing to order at least $r_i$ at each $P_i$ has codimension $\sum_{i=1}^r r_i $ in $k[x,y]_d$ (we interpret the space as empty if $\sum_{i=1}^r r_i > d$).
\end{proposition}

\begin{proof}
Vanishing at $[a_1:b_1], \ldots, [a_r:b_r]$ to orders $r_1, \ldots, r_r$ is equivalent to being divisible by $\prod_{i=1}^r (b_i x - a_i y)^{r_i}$.
\end{proof}

\subsection{Linear algebra}

Here we collect various results from linear algebra which will be needed in the proof of \Cref{thm:mainThm}.

\begin{lemma} \label{lem:signature}
Let $A = \sum_{i=1}^k \lambda_i v_i v_i^T$ be an $n \times n$ symmetric matrix, with $v_i \in \RR^n$.
If $\{v_1, \ldots, v_k \}$ are linearly independent, then the signature of $A$ is given by the sign pattern of the coefficients $\lambda_i$.
\end{lemma}

\begin{proof}
Diagonalize $A$ by extending $\{ v_1, \ldots, v_k \}$ to a basis of $\RR^n$.
\end{proof}

\begin{lemma}[Cauchy interlacing] \label{lem:cauchy_interlace} 
Let $A$ be a real symmetric matrix. If $B$ is any principal submatrix of $A$, then the eigenvalues of $B$ interlace the eigenvalues of $A$.
\end{lemma}

\begin{proof}
Cf. \cite[Theorem 4.3.17]{MR2978290}.
\end{proof}

\begin{corollary} \label{cor:psd_restriction}
Let $Q$ be a quadratic form on $\RR^n$ with Lorentz signature $(n-1, 1) = (+, \ldots, +, -)$, and $H \subseteq \RR^n$ a hyperplane. Then the following are equivalent for the restriction $Q \big|_H$ of $Q$ to $H$:

\begin{enumerate}
\item $\ker(Q \big|_H) \ne 0$
\item $\rank Q \big|_H = n - 2$
\item $Q \big|_H$ is positive semi-definite, but not positive-definite.
\end{enumerate}
\end{corollary}

\begin{proof}
Choose a basis of $\RR^n$ which arises from extending a basis of $H$, so that if $A$ is the $n \times n$ symmetric matrix representing $Q$, then $Q \big|_H$ is represented by a principal $(n-1) \times (n-1)$ submatrix $B$ of $A$.
Now $\ker(Q \big|_H) \ne 0$ implies that $0$ is an eigenvalue of $B$.
If $B$ had a negative eigenvalue, then \Cref{lem:cauchy_interlace} would imply that $A$ has $\ge 2$ negative eigenvalues, contradiction.
\end{proof}

% \begin{lemma}[Matrix determinant lemma] \label{lem:matrixDet}
% Let $A \in \RR^{n \times n}$ be invertible, and $u, v \in \RR^n$.
% Then $\det(A + uv^T) = \det(A) (1 + v^T A^{-1} u)$.
% \end{lemma}

% \begin{lemma}[Schur complement] \label{lem:schurComp}
% Let $M = \begin{bmatrix}
% A & B \\
% C & D
% \end{bmatrix}
% $ be a block matrix with $A$ invertible.
% Then $\det M = \det(A) \det(D - CA^{-1}B)$.
% \end{lemma}

\section{A Monomial Example} \label{sec:monomial}

Before proceeding to the proof of \Cref{thm:mainThm}, we illustrate the major steps of the construction in \Cref{sec:construction} in an example.
Let $C_d := \nu_d(\PP^1) \subseteq \PP^d$ be the standard rational normal curve of degree $d$, and $e_0, \ldots, e_d$ the torus-fixed points of $\PP^d$.
Set $X_i := \pi_{e_i}(C_d) \subseteq \PP^{d-1}$, the projection of $C_d$ away from $e_i$.
As this corresponds to the case where the center of projection $F = x^{d-i}y^i$ is a binary monomial of degree $d$, we refer to $X_i$ as a \emph{monomial projection} of the rational normal curve.

% \begin{remark}
% To apply \Cref{thm:mainThm}, we need some restrictions on the monomial center of projection.
% Note that $e_0 = \nu_d([1 : 0]) \in C_d$, and $e_1$ is on the tangent line to $C_d$ at $e_0$ (these are the only $e_i$ that lie on the (first) secant variety of $C_d$). 
% Since we are only interested in the case when $X_i$ is smooth, we must take $i \ge 2$.

% Moreover, if the center of projection is $e_2 = [0 : 0 : 1 : 0 : \ldots : 0] \in \PP^d$, then the corresponding binary monomial is $F(p) = x^{d-2}y^2$.
% In this case the ideal of the image curve $X \subseteq \PP^{d-1}$ contains ${d-2 \choose 2} + d - 4 = {d-1 \choose 2} - 2$ quadrics (cf. \cite[Theorem 1.3 (3)]{MR2331758}), namely 
% \[
% I_2 \left( \begin{bmatrix}
% x_0 & x_2 & \ldots & x_{d-2} \\
% x_1 & x_3 & \ldots & x_{d-1}
% \end{bmatrix} \right) + 
% \left( 
% x_2 x_i - x_1 x_{i+2} \mid 2 \le i \le d-3
% \right)
% \]
% (where the $2 \times (d-2)$ matrix above is obtained from the full $2 \times (d-1)$ Hankel matrix by deleting the second column).
% Note that these quadrics are all contained in the line $L := (x_2, \ldots, x_{d-1}) \subseteq \PP^{d-1}$ (in fact, these quadrics cut out $X \cup L$ set-theoretically).
% Thus our main result relating the Hankel index of $X$ to the almost real rank of $F(p)$ does not apply in this case.

% On the other hand, note that projecting away from $e_i$ with $3 \le i \le \lfloor \frac{d}{2} \rfloor$ gives an image curve $X$ which is contained in a surface scroll $S(a,b)$ with $a \ge 2$, so by \cite[Theorem 1.1]{MR2299577}, the ideal of $X$ is generated by quadrics.
% \end{remark}

\begin{example}
Let $d = 6$, and consider the rational normal curve $C_6 \subseteq \PP^6$. 
Set $X = X_3 = \pi_{e_3}(C_6)$, where the center of projection $e_3$ corresponds to the monomial $x^3y^3 \in \RR[x,y]_6$.
Then $(x^3y^3)^\perp = (x^4, y^4)$, and the form $x^4 - y^4 = (x-y)(x+y)(x^2+y^2)$ has almost real roots, which correspond to 4 points 
\[
\{ p_1, p_2, p_3, p_4 \} := \{ \nu_6([1:1]), \nu_6([1:-1]), \nu_6([1:i]), \nu_6([1:-i]) \} \subseteq C_6.
\]
Accordingly $\arrank(x^3y^3) = 4$, and there is a decomposition
\begin{align*}
80x^3y^3 &= (x+y)^6-(x-y)^6+i(x+iy)^6-i(x-iy)^6 \\
&= (x+y)^6-(x-y)^6+2 \Re (\zeta x+i\zeta y)^6
\end{align*}
where $\zeta := e^{\pi i /12}$.
Using notation as in  \Cref{ssec:complexPairConstruction}, we construct a ray
\begin{align*}
    \ell &:= d_1(x-y)^{12} + d_2(x+y)^{12} + d_4(\zeta x+i\zeta y)^{12} + \overline{d_4}(\overline{\zeta} x + \overline{i\zeta} y)^{12} \\
    &= d_1(x-y)^{12} + d_2(x+y)^{12} - (2\alpha+2\beta i)(x+iy)^{12} - (2\alpha-2\beta i)(x-iy)^{12}
\end{align*}
where $\alpha = \frac{\Re(d_4)}{2}, \beta = \frac{\Im(d_4)}{2}$.
We thus get a $7 \times 7$ matrix representing the quadratic form $Q_\ell$ on $\RR[x,y]_6 \cong R(C_6)_1$ corresponding to the ray $\ell$, e.g. with respect to the (normalized) monomial basis $\left\{ \dfrac{x^{6-i}y^i}{i! (6-i)!} \mid 0 \le i \le 6 \right\}$ of $\RR[x,y]_6$.

Next, we restrict $Q_\ell$ to the hyperplane $(x^3y^3)^\perp_6 \subseteq \RR[x,y]_6$, to obtain a quadratic form $q_\ell$ on $R(X)_1$.
In the chosen monomial basis, this corresponds to deleting the $4^\text{th}$ (= middle) row and column from $Q_\ell$, and yields the following block matrix:
\[
q_\ell = \begin{bmatrix} M & M \\ M & M\end{bmatrix},   \quad M := \begin{bmatrix} d_1+d_2-4\alpha & -d_1+d_2+4\beta & d_1+d_2+4\alpha \\ -d_1+d_2+4\beta & d_1+d_2+4\alpha & -d_1+d_2-4\beta \\ d_1+d_2+4\alpha & -d_1+d_2-4\beta & d_1+d_2-4\alpha \end{bmatrix}.
\]
In particular, the block structure of $q_\ell$ implies that $\rank(q_\ell) \le 2 \iff 0 = \det M = 64(d_1d_2\alpha + (d_1 + d_2) (\alpha^2 + \beta^2))$.
Thus when $\frac{\alpha}{\alpha^2+\beta^2} + (\frac{1}{d_1} + \frac{1}{d_2})= 0$, $d_1, d_2 > 0$, and $\beta \ne 0$,
the submatrix $M$ is singular and $q_\ell$ is PSD of rank $2$.
For example, if $d_1 = 1, d_2 = 1, \alpha = -1/4, \beta = 1/4$, then $M = \begin{bmatrix} 3 & 1 & 1 \\ 1 & 1 & -1 \\ 1 & -1 & 3 \end{bmatrix}$.
Thus the linear functional $\ell \in R(X)_2^\star$ has rank $2$. 
Moreover, one can compute a basis $\{ -x_0 + 12x_1 + x_6, -x_1 + x_5, -x_0 + 15x_4, -x_0 + 12x_1 + 15x_2 \}$ of $\ker(q_\ell)$ in $\RR[x_0,x_1,x_2,x_4,x_5,x_6] = R(\PP^5)$, whose vanishing defines a line in $\PP^5$.
It is readily verified that this line does not meet $X$, which shows that $\ker(q_\ell)$ is basepoint-free.
Thus $\ell \in \Sigma_X^\star \setminus P_X^\star$, hence $\eta(X) \le 2$.
As $\eta(X) \ge 2$ by definition, this shows that $\eta(X) = 2$.
\end{example}

\section{Construction of rays in \texorpdfstring{$\Sigma_X^\star$}{SigmaXstar}} \label{sec:construction}

We now turn to the proof of \Cref{thm:mainThm}, which will span the next two sections.
In this section, we give a general procedure for constructing elements in $\Sigma_X^\star$ of ranks between $\arrank(F(p)) - 2$ and $d - 3$, whose kernels are basepoint-free.
By \Cref{rem:basePtFreeVsPtEval}, this shows that if $\arrank(F(p)) > 3$, then $\eta(X) \le \arrank(F(p)) - 2$.

Choose $r$ with $\arrank(F(p)) \le r \le d-1$, and choose a form $g \in (F(p))^\perp_r$ with almost real roots. 
We assume that no proper divisor of $g$ is in $(F(p))^\perp$ (which is automatic when $r = \arrank(F(p))$, and can be arranged when $r \ge \deg F^\circ$).
Then there is a factorization over $\CC$
\[
    g =: \prod_{i=1}^r l_i
\]
of $g$ into linear forms $l_i =: a_i x + b_i y \in \CC[x,y]_1$ where either
\begin{enumerate}
    \item All $l_i$'s are distinct and real, or
    \item All $l_i$'s are distinct, and there is exactly one conjugate pair $l_r = \overline{l_{r-1}}$, or
    \item All $l_i$'s are real, and there is exactly one repeated factor $l_r = l_{r-1}$.
\end{enumerate}
For the first two cases, the construction that we give below has appeared before, e.g. in \cite[Theorem 6.1 and Theorem 7.1]{blekherman2012nonnegative} (for Veronese embeddings of projective spaces) and \cite[Proposition 3.2 and Procedure 3.3]{MR3486176}.
Case (3) however is new, specifically dealing with a non-reduced zero-dimensional scheme.

\subsection{Simple real roots} \label{ssec:simpleRootConstruction}
We start with case (1), i.e. all roots of $g$ are real and distinct.
By apolarity (\Cref{lem:genApolarity}), $F(p)$ may be expressed as a linear combination of $(l_1)_\perp^d, \ldots, (l_r)_\perp^d$, i.e. there exist $c_1, \ldots, c_r \in \RR$ such that
\begin{align} \label{eqn:linRel}
    F(p) = \sum_{i=1}^r c_i (l_i)_\perp^d
\end{align}
Note that since no proper factor of $g$ is in $(F(p))^\perp$, each coefficient $c_i$ in (\ref{eqn:linRel}) is nonzero.

We now construct elements in $\Sigma_X^\star$ of rank $r-2$.
Let $p_1, \ldots, p_r \in \PP^{d}$ correspond to the $r$ roots of $g$ (explicitly, $p_i = \nu_d([a_i : b_i])$).
Consider a linear combination
\begin{align} \label{eqn:constructedRay}
    \ell := \sum_{i=1}^r d_i \ell_{p_i}^2 \in (R(C_d)_2)^\star
\end{align}
with (as yet unspecified) coefficients $d_i \in \RR$, where $\ell_{p_i} =$ evaluation at $p_i$ (note that $\ell_{p_i}$ corresponds to the binary form $(l_i)_\perp^{d} \in \RR[x,y]_{d}$).
Then as in \Cref{rem:identifications}, $\ell$ gives rise to a quadratic form $Q_\ell$ on $R(C_d)_1$, as well as its restriction $q_\ell$ to $R(X)_1$.

Next, we claim that if the $d_i$ are chosen so that
\begin{align} \label{eqn:relForKer}
    d_1, \ldots, d_{r-1} > 0, \quad \sum_{i=1}^r \dfrac{c_i^2}{d_i} = 0,
\end{align}
then $\rank(q_\ell) = r - 2$.
To show this, we choose coordinates to reduce to a computation with matrices.
Let
\[
Z := \{ p_1, \ldots, p_r \} \subseteq C_d
\]
be the zero-dimensional variety of the points $p_i$. 
The coordinate ring $R(Z)$ satisfies $\dim_{\RR} R(Z)_1 = r$, with basis $\{e_i\}_{i=1}^r$ given by indicator functions of the points, i.e. $e_i(p_j) = \delta_{ij}$.
(One can of course write down explicit polynomial representatives on $\PP^d$ for the $e_i$'s via interpolators (with a suitable padding up to degree $d$), although we will not need such representatives.)
If $I(Z)$ is the defining ideal of $Z$ in $C_d$, then via the isomorphism $R(Z) \cong R(C_d)/I(Z)$, a quadratic form on $R(C_d)_1$ whose kernel contains $I(Z)_1$ (such as $Q_\ell$) induces a quadratic form on $R(Z)_1$, which is in turn represented as an $r \times r$ matrix.

The choice of basis $\{e_i\}_{i=1}^r$ then allows for a convenient expression of the matrix of the induced quadratic form $\widetilde{Q_\ell}$ on $R(Z)_1$: namely, $\widetilde{Q_\ell}$ is represented by a diagonal matrix $\diag(d_1, \ldots, d_r)$ in this basis.
% We now express the quadratic form $Q_\ell$ as a matrix in the basis $\{e_i\}_{i=1}^r$.
Note that the conditions (\ref{eqn:relForKer}) imply that $d_r < 0$ (recall that $c_i \ne 0$), so by \Cref{lem:signature}, $\widetilde{Q_\ell}$ has Lorentz signature (since $r \le d+1$, any set of $r$ points on $C_d$ are in linearly general position, so the functionals $\ell_{p_1}, \ldots, \ell_{p_r} \in (R(C_d)_2)^\star$ are linearly independent).

On the other hand, we may also consider the quadratic form induced by $q_\ell$ on the points $\pi(Z) := \{ \pi(p_1), \ldots, \pi(p_r) \}$.
The key difference is that the points $\pi(p_1), \ldots, \pi(p_r) \in X$ are \emph{not} in linearly general position -- indeed, the projection map $\pi$ can be viewed as a projectivization of the vector space quotient $\RR[x,y]_d \twoheadrightarrow \RR[x,y]_d/\Span\{F(p)\}$, so the linear relation (\ref{eqn:linRel}) gives a linear dependency
\begin{equation} \label{eqn:linRelProj}
    0 = \sum_{i=1}^r c_i \ell_{\pi(p_i)}
    \quad \iff \quad
    \ell_{\pi(p_r)} = -\dfrac{1}{c_r} \sum_{i=1}^{r-1} c_i \ell_{\pi(p_i)}
\end{equation}
expressing the last point evaluation $\ell_{\pi(p_r)}$ in terms of the others.
In particular, on removing the last point $\pi(p_r)$, the coordinate ring $R(\pi(Z \setminus \{ p_r \}))$ has a basis $\{e_i\}_{i=1}^{r-1}$ for its degree 1 part (note that $\pi(Z \setminus \{ p_r \})$ \emph{is} in linearly general position in $\PP^{d-1}$).
This gives a quadratic form $\widetilde{q_\ell}$ on $R(\pi(Z \setminus \{ p_r \}))_1$ induced by $q_\ell$: explicitly, substituting (\ref{eqn:linRelProj}) into (\ref{eqn:constructedRay}) gives the expression
$\displaystyle \sum_{i=1}^{r-1} d_i \ell_{\pi(p_i)}^2 + \frac{d_r}{c_r^2} \Big( \sum_{i=1}^{r-1} c_i \ell_{\pi(p_i)} \Big)^2$ for (the linear functional corresponding to) $\widetilde{q_\ell}$. 
Setting 
\[
    D := \diag(d_1, \ldots, d_{r-1}), \quad
    \bc := \begin{bmatrix} 
    c_1 & \ldots & c_{r-1} 
    \end{bmatrix}^T,
\]
we see that the matrix of $\widetilde{q_\ell}$ in the basis $\{e_i\}_{i=1}^{r-1}$ is given by $D + \dfrac{d_r}{c_r^2} \bc \bc^T$. Finally, observe that the vector $D^{-1}\bc$ is in the kernel of $\widetilde{q_\ell}$:
\begin{align*}
    (D + \dfrac{d_r}{c_r^2} \bc \bc^T)(D^{-1}\bc) &= \bc + \dfrac{d_r}{c_r^2} \bc \bc^T D^{-1} \bc \\
    &= \bc \left( 1 + \dfrac{d_r}{c_r^2} \left( \sum_{i=1}^{r-1} \dfrac{c_i^2}{d_i} \right) \right) = 0
\end{align*}
by (\ref{eqn:relForKer}).
\Cref{cor:psd_restriction} then implies that $q_\ell$ is PSD (which implies that $\ell \in \Sigma_X^\star$) of rank $r-2$.

% To conclude that $\eta(X) \le r - 2$, 
It remains to show that for any ray $\ell$ constructed satisfying (\ref{eqn:constructedRay}) and (\ref{eqn:relForKer}), $\ker(q_\ell)$ is basepoint-free. The following reasoning will also apply to the cases in \Cref{ssec:complexPairConstruction} and \Cref{ssec:doubleRootConstruction}.
First, we claim that $\ker(q_\ell)$ can have no basepoints outside of $\pi(Z)$: if not, then $\ker(Q_\ell)$ would have a basepoint outside of $Z$.
However, by \Cref{lem:kerQuadFormViaApolar}, $\ker(Q_\ell) = (g)_d \subseteq (F(p))^\perp_d$ 
%\jc{check} \gb{This argument is potentially not good. The kernel $\ker Q_\ell$ may be larger (because we restrict to hyperplane), and so forms in $\ker q_\ell$ may have extra vanishing)} \gb{the key point is the following...}{\color{red} We are restricting to the hyperplane $H$ orthogonal to the center of projection. The kernel of $q_\ell$ contains all forms in $H$ that vanish on $Z$. Let's call this vector space $B$. The question is whether polynomials in $B$ have additional zeroes outside of $Z$. Let $Z'$ be $Z$ minus a simple point of $Z$. Polynomials in $H$ vanishing on $Z'$ already vanish on all of $Z$. Therefore an additional point of vanishing is not possible, since then vanishing on $Z'$ would imply vanishing on two points in the hyperplane $H$.}  
is an $\RR$-vector space of dimension $d+1 - \deg g = d+1-r$ which consists of binary forms vanishing at all the points of $Z$ (to orders specified by multiplicities of factors of $g$ in the case of a double root in \Cref{ssec:doubleRootConstruction}), thus cannot have another common zero outside of $Z$ by \Cref{prop:vanishingIndepCond}.
It thus suffices to eliminate the possibility of any point of $\pi(Z)$ as a basepoint, but this follows since the vector $D^{-1}\bc$ in $\ker(q_\ell)$ has all nonzero entries in the basis $\{e_i\}_{i=1}^{r-1}$. 

\subsection{One complex pair} \label{ssec:complexPairConstruction}
Next, we consider case (2), i.e. $g$ has one pair of nonreal roots $l_r = \overline{l_{r-1}}$.
The general argument will follow the outline of case (1), so we focus only on the differences (which will mainly be in the last two functionals).
Essentially, rather than using two functionals arising from evaluations at complex conjugate points, we use the real and imaginary parts of one complex point evaluation.
Over $\CC$, there is an expression $F(p) = \sum_{i=1}^{r-2} c_i (l_i)_\perp^d + c_{r-1} (l_{r-1})_\perp^d + c_r (l_r)_\perp^d$, and independence of the forms $\{ (l_i)_\perp^d \}_{i=1}^r$ and conjugate-symmetry forces $c_r = \overline{c_{r-1}}$.
By rescaling $l_r \in \CC[x,y]_1$ we may assume that $c_r = 1$ (so that $c_{r-1} = 1$ as well), and thus write the analogue of (\ref{eqn:linRel}) in the form 
\begin{equation} \tag{$3'$} \label{eqn:linRel'}
F(p) = \sum_{i=1}^{r-2} c_i (l_i)_\perp^d + 2 \Re((l_r)_\perp^d)
\end{equation}
where $c_1, \ldots, c_{r-2} \in \RR$ are all nonzero, since no proper factor of $g$ is in $(F(p))^\perp$.

We then construct the functional in $\Sigma_X^\star$.
As before, choosing $p_1, \ldots, p_r \in \PP^d$ corresponding to the roots of $g$ (with $p_r = \overline{p_{r-1}}$ a nonreal conjugate pair), we obtain a linear functional $\ell := \sum_{i=1}^{r-2} d_i \ell_{p_i}^2 + d_r \ell_{p_r}^2 + \overline{d_r} \overline{\ell_{p_r}}^2 \in R(C_d)_2^\star$, which becomes

\begin{equation} \tag{$4'$} \label{eqn:constructedRay'}
    \ell := \sum_{i=1}^{r-2} d_i \ell_{p_i}^2 + 4 \alpha (\Re((l_r)_\perp^d)^2 - \Im((l_r)_\perp^d)^2) - 8 \beta (\Re((l_r)_\perp^d) \Im((l_r)_\perp^d))
\end{equation}
where $\alpha := \frac{\Re(d_r)}{2}$, $\beta := \frac{\Im(d_r)}{2}$.
We claim that if the $d_i$ are chosen so that
\begin{equation} \tag{$5'$} \label{eqn:relForKer'}
    d_1, \ldots, d_{r-2} > 0, \; \beta \ne 0, \; \frac{\alpha}{\alpha^2 + \beta^2} + \sum_{i=1}^{r-2} \frac{c_i^2}{d_i} = 0,
\end{equation}
then $q_\ell$ has rank $r-2$ and basepoint-free kernel.
Indeed, writing $\ell_1, \ldots, \ell_r$ for the images of $(l_1)_\perp^d, \ldots, (l_{r-2})_\perp^d$, $2 \Re((l_r)_\perp^d), 2 \Im((l_r)_\perp^d)$ in $R(C_d)_1^\star$, 
and choosing forms in $R(C_d)_1$ dual to the functionals $\ell_1, \ldots, \ell_r$, we see that the matrix of $Q_\ell \Big|_{\Span\{e_i\}}$ is given by $\begin{bmatrix} D & 0 \\ 0 & A \end{bmatrix}$ where $D := \diag(d_1, \ldots, d_{r-2})$, $A := \begin{bmatrix} \alpha & -\beta \\ -\beta & -\alpha \end{bmatrix}$, so that $Q_\ell$ has Lorentz signature (note that $\det(A) < 0$).
Expressing (\ref{eqn:linRel'}) in the form
\begin{equation} \tag{$6'$} \label{eqn:linRelProj'}
    \ell_{r-1} = -\sum_{i=1}^{r-2} c_i \ell_i,
\end{equation}
setting $\bc := \begin{bmatrix} c_1 & \ldots & c_{r-2} \end{bmatrix}^T$, and substituting (\ref{eqn:linRelProj'}) into (\ref{eqn:constructedRay'}) gives the matrix
\begin{align*}
\widetilde{q_\ell} = \begin{bmatrix}
    D+\alpha \bc \bc^T & \beta \bc \\
   \beta \bc^T & -\alpha
\end{bmatrix}.
\end{align*}
Finally, observe that the vector $\begin{bmatrix}
D^{-1} \bc \\ \frac{-\beta}{\alpha^2+\beta^2}
\end{bmatrix}$ is in $\ker \widetilde{q_\ell}$, and has all nonzero entries:
\begin{align*}
    \begin{bmatrix}
    D+\alpha \bc \bc^T & \beta \bc \\
   \beta \bc^T & -\alpha
\end{bmatrix} \begin{bmatrix}
D^{-1} \bc \\ \frac{-\beta}{\alpha^2+\beta^2}
\end{bmatrix} &= \begin{bmatrix}
\bc + \alpha \bc \bc^T D^{-1} \bc - \frac{\beta^2}{\alpha^2 + \beta^2} \bc \\ 
\beta \bc^T D^{-1} \bc + \frac{\alpha \beta}{\alpha^2+\beta^2}
\end{bmatrix} \\
&= \begin{bmatrix}
\bc (1 + \alpha \bc^T D^{-1} \bc - \frac{\beta^2}{\alpha^2 + \beta^2}) \\
\beta (\bc^T D^{-1} \bc + \frac{\alpha}{\alpha^2+\beta^2})
\end{bmatrix} \\
&= \begin{bmatrix}
\alpha \bc (\bc^T D^{-1} \bc + \frac{\alpha}{\alpha^2 + \beta^2}) \\
\beta (\bc^T D^{-1} \bc + \frac{\alpha}{\alpha^2+\beta^2})
\end{bmatrix} = 0.
\end{align*}
The reasoning that $\ker(q_\ell)$ is basepoint-free was already explained at the end of \Cref{ssec:simpleRootConstruction}.
%\gb{We are not doing bpf argument here... and the next section...}
\subsection{One double root} \label{ssec:doubleRootConstruction}
Finally, we consider case (3), i.e. $g$ has a unique real double root $l_r = l_{r-1}$ (with all other roots real and simple).
In this case, the two functionals we use correspond to evaluation at the double point, as well as differentiation followed by evaluation.
From apolarity, there is a relation 
\begin{equation} \tag{$3''$} \label{eqn:linRel''}
    F(p) = \sum_{i=1}^{r-2} c_i (l_i)_\perp^d + c_{r-1} (l_r)_\perp^d + c_r l_r (l_r)_\perp^{d-1}
\end{equation}
where as before $c_1, \ldots, c_r \in \RR$ are all nonzero.
Let $\ell_1, \ldots, \ell_r \in R(C_d)_1^\star$ be the linear functionals corresponding to $(l_1)_\perp^d, \ldots, (l_{r-2})_\perp^d, (l_r)_\perp^d, 2 l_r (l_r)_\perp^{d-1}$, and consider the linear functional in $(R(C_d)_2)^\star$ defined by
\begin{equation} \tag{$4''$} \label{eqn:constructedRay''}
\ell := \sum_{i=1}^{r-1} d_i \ell_i^2 + d_r \ell_{r-1} \ell_r.
\end{equation}
We claim that if the $d_i$ are chosen so that
\begin{equation} \tag{$5''$} \label{eqn:relForKer''}
d_1, \ldots, d_{r-1} > 0, \; d_{r-1} - 2 \frac{d_r c_{r-1}}{c_r} - \frac{d_r^2}{c_r^2} \sum_{j=1}^{r-2} \frac{c_j^2}{d_j}=0
\end{equation}
then $q_\ell$ has rank $r-2$ and basepoint-free kernel.
Indeed, the matrix of $Q_\ell$ (restricted to the subspace of $R(C_d)_1$ spanned by forms dual to $\ell_1, \ldots, \ell_r$) is given by $\begin{bmatrix} D & 0 \\ 0 & A \end{bmatrix}$ where $D := \diag(d_1, \ldots, d_{r-2})$, $A := \begin{bmatrix} d_{r-1} & \frac{d_r}{2} \\ \frac{d_r}{2} & 0 \end{bmatrix}$, hence has Lorentz signature (note that $\det(A) < 0$).
Writing (\ref{eqn:linRel''}) in the form
\begin{equation} \tag{$6''$} \label{eqn:linRelProj''}
\ell_r = -\frac{2}{c_r} \sum_{i=1}^{r-1} c_i \ell_i,
\end{equation}
setting $\bc := \begin{bmatrix} c_1 & \ldots & c_{r-2} \end{bmatrix}^T$, and substituting (\ref{eqn:linRelProj''}) into (\ref{eqn:constructedRay''}) gives the matrix 
\begin{align*}
\widetilde{q_\ell} = \begin{bmatrix}
D & -\frac{d_r}{c_r} \bc \\
   -\frac{d_r}{c_r} \bc^{T} & d_{r-1} - \frac{2 d_r c_{r-1}}{c_r}
\end{bmatrix}.
\end{align*}
As before, we exhibit a kernel vector $\begin{bmatrix} D^{-1} \bc \\ \frac{c_r}{d_r} \end{bmatrix}$ with all nonzero entries:
\begin{align*}
\begin{bmatrix} D & -\frac{d_r}{c_r} \bc \\
      -\frac{d_r}{c_r} \bc^{T} & d_{r-1} - \frac{2 d_r c_{r-1}}{c_r}
      \end{bmatrix} \begin{bmatrix} D^{-1} \bc \\ \frac{c_r}{d_r} \end{bmatrix} 
   &= \begin{bmatrix} \bc - \bc \\ -\frac{d_r}{c_r} \bc^T D^{-1} \bc + \frac{c_r}{d_r}(d_{r-1} - \frac{2 d_r c_{r-1}}{c_r}) \end{bmatrix} \\
   &= \begin{bmatrix} 0 \\ \frac{c_r}{d_r}(d_{r-1} - 2 \frac{d_r c_{r-1}}{c_r} - \frac{d_r^2}{c_r^2} \bc^T D^{-1} \bc) \end{bmatrix} = 0.
\end{align*}
We remark that the (LHS of the) equation in (\ref{eqn:relForKer''}) is precisely the Schur complement $\widetilde{q_\ell}/D$ (this provides another proof that $\widetilde{q_\ell}$ is PSD but not PD). As a quadratic in $\frac{d_r}{c_r}$, this equation always has 2 real solutions (as the discriminant $(2c_{r-1})^2 + 4d_{r-1} \bc^T D^{-1} \bc$ is $> 0$ by (\ref{eqn:relForKer''})).

As before, the reasoning that $\ker(q_\ell)$ is basepoint-free was given at the end of \Cref{ssec:simpleRootConstruction}.

\section{Lower bound} \label{sec:lowerBound}

In this section, we prove a lower bound on the Hankel index in terms of the almost real rank of the center of projection, showing that our construction in \Cref{sec:construction} of rays of minimal rank is sharp.
Throughout, let $X = \pi_p(C_d)$ be a projection with center $p$ of a rational normal curve $C_d \subseteq \PP^d$.
We assume that the center $p$ is not contained in $C_d^3$ (which implies $d \ge 6$), and as in \Cref{ssec:pointsToForms}, we associate to $p$ a binary form $F(p) \in \RR[x,y]_d$.

\begin{theorem} \label{thm:lowerBound}
We have the following bound on the Hankel index of a projected rational normal curve $X$ with center of projection $F(p)$:
$$\eta(X)\geq \arrank(F(p)) - 2.$$
\end{theorem}

\begin{proof}
Fix a ray $\ell \in \Sigma_X^\star$.
By \Cref{rem:identifications}(v), we get $L \in \RR[x,y]_{2d}$ such that $\ell(\cdot) = \langle \cdot, L \rangle$, and quadratic forms $Q_\ell$ on $\RR[x,y]_d$ and $q_\ell := Q_\ell \Big|_H$ where $H = (F(p))^\perp_d \cong R(X)_1$.
Note that $q_\ell$ is PSD since $\ell$ was an element of $\Sigma_X^\star$, which implies by \Cref{lem:cauchy_interlace} that $Q_\ell$ has at most one negative eigenvalue.

We now further assume that $\ker q_\ell$ is basepoint-free.
By \Cref{lem:signature}, this implies that $Q_\ell$ is not PSD, as otherwise $q_\ell$ would be a sum of point evaluations, contradicting \Cref{rem:basePtFreeVsPtEval}.
It follows that $Q_\ell$ has Lorentz signature $(+, \ldots, +, -)$.

Consider the apolar ideal $(L)^\perp = (L_\perp, L^\circ)$, and set $s := \deg L_\perp$ (so that $\deg L^\circ = 2d + 2 - s$).
By \Cref{lem:kerQuadFormViaApolar}, $\ker(Q_\ell) = (L)^\perp_d$, and since this space is nonzero (being basepoint-free), one must have $s \le d$ (in particular, $s < \deg L^\circ$).
Write 
\[
L_\perp := \prod_{i=1}^t l_i^{d_i}
\]
where $l_i \in \RR[x,y]_1$ are distinct linear forms and $\sum d_i = s$.

We next claim that $L_\perp$ has almost real roots, which is the core of this proof.
For convenience, say that a form $G$ has a \emph{triple root} if $G$ has a real root of multiplicity $3$, and all other roots are real and simple.
We first show, via a perturbation argument, that either $L_\perp$ has almost real roots, or $L_\perp$ has a triple root.
Then, we show that $L_\perp$ does not have a triple root.

Thus, suppose that $L_\perp$ does not have almost real roots, nor a triple root.
The key idea for the perturbation argument is the following: we may approximate $L_\perp$ by a sequence of polynomials, all of which have at least 2 pairs of simple complex roots.
Intuitively, each pair of simple complex roots contributes a negative eigenvalue to the signature, and then continuity implies that $Q_\ell$ has $\ge 2$ negative eigenvalues, a contradiction.

To be precise, we consider the following types of replacements of certain factors of $L_\perp$, depending on the way that $L_\perp$ fails to have almost real roots/a triple root:
\begin{align*}
    (x - \alpha)^2 (x - \overline{\alpha})^2 &\longrightarrow (x - \alpha - \epsilon) (x - \overline{\alpha} - \epsilon) (x - \alpha + \epsilon) (x - \overline{\alpha} + \epsilon) \\
    (x - a)^4 &\longrightarrow (x - a)^4 + \epsilon^4 \\
    (x - a)^2 (x-b)^2 &\longrightarrow ((x - a)^2 + \epsilon^2) ((x - b)^2 + \epsilon^2) \\
    (x - \alpha)(x - \overline{\alpha})(x - a)^2 &\longrightarrow (x - \alpha) (x - \overline{\alpha}) ((x - a)^2 + \epsilon^2)
\end{align*}
(here $\alpha \in \CC \setminus \RR$, and $a, b \in \RR$ are distinct).
Then for all sufficiently small $\epsilon > 0$, the polynomial $L_\epsilon$ obtained from $L_\perp$ by performing one of the above replacements has $\ge 2$ pairs of simple complex roots, and satisfies $L_\epsilon \to L_\perp$ as $\epsilon \to 0$ (if $L_\perp$ already has 2 pairs of simple complex roots, then we may take $L_\epsilon = L_\perp$).
Taking apolar ideals of the form $(L_\epsilon, L^\circ)$ gives a sequence of degree $2d$ forms converging to $L$, and with this associated quadratic forms $Q_\epsilon \to Q_\ell$.
Then each $L_\epsilon$ has $\ge 2$ pairs of simple complex roots, so $Q_\epsilon$ has $\ge 2$ negative eigenvalues.
Furthermore, $\dim \ker(Q_\epsilon) = \dim (L_\epsilon)_d = d-s+1$ is constant in $\epsilon$.
Then continuity of eigenvalues implies that $Q_\ell$ has $\ge 2$ negative eigenvalues (as no negative eigenvalue can become positive without crossing zero, and the number of zero eigenvalues stays constant), contradicting the fact that $Q_\ell$ has Lorentz signature. 

To conclude that $L_\perp$ has almost real roots, it remains to eliminate the possibility that $L_\perp$ has a triple root.
We will show that if $L_\perp$ has a triple root, then $\ker q_\ell$ is not basepoint-free.
Suppose the roots of $L_\perp$ have multiplicities $(d_1, \ldots, d_{s-2}) = (1, \ldots, 1, 3)$.
Setting $l := l_{s-2}$, by apolarity we may write
\begin{equation} \label{eqn:tripleRootRay}
    L = \sum_{i=1}^{s-3} d_i (l_i)_\perp^{2d} + d_{s-2} (l_\perp)^{2d} + d_{s-1} l(l_\perp)^{2d-1} + d_s l^2 (l_\perp)^{2d-2}
\end{equation}
for some $d_1, \ldots, d_s \in \RR$.
Write $\ell_1, \ldots, \ell_s$ for the functionals in $R(C_d)_1^\star$ corresponding to $(l_1)_\perp^d, \ldots, (l_{s-3})_\perp^d$, $(l_\perp)^d, l(l_\perp)^{d-1}, l^2(l_\perp)^{d-2}$.
Then (\ref{eqn:tripleRootRay}) may be expressed as 
\begin{equation} \label{eqn:tripleRootRay'}
    \ell = \sum_{i=1}^{s-2} d_i \ell_i^2 + d_{s-1} \ell_{s-2} \ell_{s-1} + d_s \ell_{s-2} \ell_s
\end{equation}
(note that $\ell_{s-2} \ell_s = \ell_{s-1}^2$). Since $L_\perp \in H = (F(p))^\perp$ (shown below), there is also a relation
\[
0 = \sum_{i=1}^s c_i \ell_i
\]
with $c_i \in \RR$.
Note that since a proper factor of $L_\perp$ may lie in $(F(p))^\perp$, we cannot say a priori whether any particular $c_i$ is nonzero.
We thus consider cases depending on whether $c_s$ is nonzero.

If $c_s \ne 0$, then substituting $\ell_s = -\frac{1}{c_s} \sum_{i=1}^{s-1} c_i \ell_i$ into (\ref{eqn:tripleRootRay'}) gives a matrix for $\widetilde{q_\ell}$ (with respect to $\{\ell_1, \ldots, \ell_{s-1}\}$) whose last diagonal entry ($=$ coefficient of $\ell_{s-1}^2)$) is 0.
If $c_s = 0$, then substituting $\ell_i = -\frac{1}{c_i} \sum_{j \ne i}^{s-1} c_j \ell_j$ (for some $1 \le i \le s-1$) into (\ref{eqn:tripleRootRay'}) gives a matrix for $\widetilde{q_\ell}$ (with respect to $\{\ell_1, \ldots, \hat{\ell_i}, \ldots, \ell_s \}$) whose last diagonal entry ($=$ coefficient of $\ell_s^2)$) is 0.
Thus in any case $\widetilde{q_\ell}$ can be represented by a matrix with last diagonal entry $0$, and since $\widetilde{q_\ell}$ is PSD, this implies that the entire last column of $\widetilde{q_\ell}$ must be $0$.
Then $\ker(\widetilde{q_\ell})$ is generated by the vector $\begin{bmatrix} 0 & \ldots & 0 & 1 \end{bmatrix}^T$, but this implies that $\ker(q_\ell)$ is not basepoint-free (as each of the roots of $l_1, \ldots, l_{s-1}$ would be basepoints).

This shows that $L_\perp$ has almost real roots.
Next, we show that $L_\perp$ is contained in the apolar ideal of the center $(F(p))^\perp$.
If $L_\perp$ has simple roots, then the points on $X$ corresponding to these roots cannot be in linearly general position: if they were, then \Cref{lem:signature} implies that $\ell$ would be a sum of point evaluations, contradicting \Cref{rem:basePtFreeVsPtEval}(iii).
This means precisely that $F(p)$ can be written as a linear combination of $d^\text{th}$ powers of roots of $L_\perp$, so by apolarity $L_\perp \in (F(p))^\perp$.

Next, suppose that $L_\perp$ does not have simple roots, and define the following ``reduction of order'' polynomial
\[
\widetilde{L_\perp} := \prod_{i=1}^t l_i^{\lceil d_i/2 \rceil}
\]
with the key property that $L_\perp$ divides $\widetilde{L_\perp}^2$.
We claim that 
\begin{align} \label{eqn:tildeEquality}
    (\widetilde{L_\perp})_d \cap H = \ker(q_\ell) = (L_\perp)_d \cap H.
\end{align}
To see this, note that for $f \in H$, one has $f \in \ker(q_\ell) \iff q_\ell(f) = 0$ (as $q_\ell$ is PSD on $H$ -- this need not be the case if $q_\ell$ were indefinite). 
Together with \Cref{lem:kerQuadFormViaApolar}, this gives the second equality.
For the first equality, note that $L_\perp \in (\widetilde{L_\perp}) \implies (L_\perp)_d \cap H \subseteq (\widetilde{L_\perp})_d \cap H$.
Conversely, any $f \in (\widetilde{L_\perp})_d \cap H$ is of the form $f := g \widetilde{L_\perp}$ (for some $g \in \RR[x,y]_{d-\deg \widetilde{L_\perp}}$), hence satisfies $q_\ell(f) = \langle g^2(\widetilde{L_\perp})^2, L \rangle = 0$, since $L_\perp$ divides $\widetilde{L_\perp}^2$, and $\langle L_\perp, L \rangle = 0$.

In view of (\ref{eqn:tildeEquality}): given that $L_\perp \ne \widetilde{L_\perp}$, one has $(L_\perp)_d \subsetneq (\widetilde{L_\perp})_d$, but since the intersections of these subspaces with the hyperplane $H$ coincide, it must be the case that $\dim (L_\perp)_d$, $\dim (\widetilde{L_\perp})_d$ differ by exactly $1$ (note that dimension decreases by at most $1$ when intersecting with a hyperplane, and does not change precisely when the subspace is already contained in the hyperplane).
From this we deduce that $(L_\perp)_d \subseteq H$, hence $L_\perp \in (F(p))^\perp$ by \Cref{lem:apolarMembershipSameDegree}.
(Note that this argument also shows that $\deg L_\perp \le 1 + \deg \widetilde{L_\perp}$, which gives another proof that $L_\perp$ has at most one multiple real root, which must be of multiplicity $\le 3$).

Putting the above results together, we see that $L_\perp \in (F(p))^\perp$ has almost real roots, so $\arrank(F(p)) \le \deg L_\perp = s$.
Now $\dim \ker(Q_\ell) = \dim (L)_d^\perp = \dim (L_\perp)_d = d-s+1$, so $\rank(Q_\ell) = d+1 - \dim \ker(Q_\ell) = s$, and by \Cref{cor:psd_restriction}, $\rank(q_\ell) = \rank(Q_\ell) - 2$.
Thus $\rank(\ell) = \rank(q_\ell) = s - 2 \ge \arrank(F(p)) - 2$.
Since this holds for any ray $\ell$ with $\ker(q_\ell)$ basepoint-free, in particular it holds for any extreme ray of $\Sigma_X^\star$ which is not a point evaluation, so $\eta(X) \ge \arrank(F(p)) - 2$ as desired.
\end{proof}

\section{Almost real rank} \label{sec:almostRealRank}

As shown by our main result \Cref{thm:mainThm}, the almost real rank of a form is an interesting quantity to study.
In this final section, we investigate almost real rank of binary forms in general.
To begin, the following proposition characterizes some cases where the almost real rank is small.

\begin{proposition} \label{prop:arrankBasicProp}
Let $d \ge 3$ and $F \in \RR[x,y]_d$, with apolar ideal $(F)^\perp = (F_\perp, F^\circ)$ of type $(d_1, d_2)$.

\begin{enumerate}
    \item $\arrank(F) = d_1 \iff F_\perp$ has almost real roots.
    \item If $\arrank(F) > d_1$, then $\arrank(F) \ge d_2$.
    \item $\arrank(F) = 1 \iff d_1 = 1 \iff \rrank(F) = 1$.
    \item $\arrank(F) = 2 \iff d_1 = 2 \iff \cbrank(F) = 2$.
    \item $\arrank(F) = 3 \iff d_1 = 3$ and $F_\perp$ is not a cube (of a linear form).
\end{enumerate}
(If $d_1 = d_2$, we interpret ``$F_\perp$ has almost real roots'' to mean ``there exists a form in $(F)^\perp_{d_1}$ with almost real roots'', and similarly in (5)).
\end{proposition}

\begin{proof}
Omitted.
\end{proof}

\begin{remark} \label{rem:stratification}
One can stratify all degree $d$ binary forms by almost real rank as follows:
write $V_i := H^0(\mathcal{O}_{\PP^1}(i))$ for the vector space of (real) degree $i$ binary forms.
Let $\varphi_{1,d} : \PP(V_1) \to \PP(V_d)$ be the $d^\text{th}$ Veronese map, and for $r \ge 2$, define the map

\begin{align*}
\varphi_{r,d} : \PP(V_2) \times (\PP(V_1))^{r-2} &\times \PP^{r-1} \to \PP(V_d) \\
(q, (l_j)_{j=0}^{r-3}, [c_0 : \ldots : c_{r-1}]) \mapsto &\sum_{j=0}^{r-3} c_j l_j^d + c_{r-2}q_1 + c_{r-1}q_2
\end{align*}
(here $q_1, q_2$ are the degree $d$ forms corresponding to the complex linear factors of the quadric $q$ as in \Cref{sec:construction}, e.g. if $q = l^2$, then $q_1 = l^d$, $q_2 = l_\perp(l)^{d-1}$).
By \Cref{lem:genApolarity}, the image of $\varphi_{r,d}$ is precisely the set of degree $d$ binary forms of almost real rank $\le r$.
Restricting $\varphi_{r,d}$ to the (open) subset where $q, l_0, \ldots, l_{r-3}$ are relatively prime, and removing the image of $\varphi_{r-1, d}$, gives the set of degree $d$ binary forms of almost real rank $= r$.

From this description, one can deduce various structural properties of the set of forms of a given almost real rank.
For instance, $\varphi_{2,d}$ is injective (for $d \ge 3$), so the set of forms with almost real rank $\le 2$ has dimension $3$.
Also, when $r = \lfloor \frac{d+2}{2} \rfloor = \lfloor \frac{d}{2} \rfloor + 1$, $\varphi_{r,d}$ is dominant, corresponding to the fact that the generic type is $(r, d+2-r)$, and among forms of degree $r$, those with almost real roots are typical.
For dimension reasons, this is the least value of $r$ for which $\varphi_{r,d}$ can be dominant, with general fibers of dimension $0$ (resp. $1$) when $d$ is odd (resp. even).
\end{remark}

It is natural to ask what the maximal almost real rank is for binary forms of degree $d$.
This is answered by the next theorem:

\begin{theorem} \label{thm:maxARrank}
For any $d \ge 3$ and $F \in \RR[x,y]_d$, $\arrank(F) \le d-1$.
\end{theorem}

\begin{proof}
First we reduce to the case that $1<\crank(F) < d$.
If $\crank(F) = d$, then the apolar ideal $(F)^\perp$ is of type $(2, d)$ (cf. \Cref{rem:ranksViaApolar}), so $\arrank(F) = 2 \le d-1$ by \Cref{prop:arrankBasicProp}(4).
Additionally, if $\crank(F) = 1$, then $\arrank(F) = 1$ as well.
Thus we may assume $2 \le \crank(F) \le d-1$.

We now induct on $d$.
For the base case $d = 3$, the apolar ideal is of type $(2,3)$, so again $\arrank(F) \le 2$.
For the inductive step, choose any direction $u = (u_1, u_2) \in \RR^2$, corresponding to a linear form $l_u(x,y) := u_1x + u_2y$.
Then by induction, the apolar ideal of the directional derivative $D_u(F) = \langle l_u, F \rangle$ contains a form with almost real roots of degree $\le d-2$ (note that $D_u(F) \ne 0$, since $\crank(F) > 1$ by assumption $\implies l_u \not \in (F)^\perp$).
By multiplying an additional factor if necessary, we may choose $G \in (D_u(F))^\perp$ of degree $= d-2$ with almost real roots.
Then $G \cdot l_u \in (F)^\perp$ is of degree $d-1$.
Since $\crank(F) \le d-1$, we may also choose $H \in (F)^\perp$ of degree $= d-1$ with simple complex roots.

We claim that for sufficiently small $\epsilon \in \RR$, the form
$G_\epsilon := G \cdot l_u + \epsilon H \in (F)^\perp$ has almost real roots.
First, observe that there are only finitely many $\epsilon$ such that $G_\epsilon$ does not have simple roots: these are given by the roots of the discriminant of $G_\epsilon$, viewed as a polynomial in $\epsilon$ (note that this polynomial is nonzero, since $H$ has simple roots).
Thus by avoiding these finitely many choices of $\epsilon$, we may assume that $G_\epsilon$ has simple roots, and so it suffices to show that $G_\epsilon$ has at most $1$ pair of complex roots.

For $|\epsilon|$ sufficiently small, any simple root of $G \cdot l_u$ gives a simple root of $G_\epsilon$ (by dehomogenizing we may consider a simple root of a univariate real polynomial, which is e.g. negative to the left of the root and positive to the right, and this is stable under small perturbation). 
Thus we need only consider the following cases: (i) $G \cdot l_u$ has a triple root, and (ii) $G \cdot l_u$ has 2 double roots.
In case (i), since $G_\epsilon$ has simple roots, the triple root of $G \cdot l_u$ induces either 3 distinct real roots of $G_\epsilon$, or 1 real root and 1 complex pair, and since all other roots of $G \cdot l_u$ are real and simple in this case, we get at most $1$ pair of complex roots of $G_\epsilon$.

In case (ii), suppose $G \cdot l_u$ has 2 double roots, and let $p$ be one of these.
If $p$ is a root of $H$, then $p$ is also a root of $G_\epsilon$ for any $\epsilon$, so the double root $p$ of $G \cdot l_u$ induces 2 real roots of $G_\epsilon$ (one of which is $p$, which implies that the other root must be real).
Otherwise, if $p$ is not a root of $H$, then $G \cdot l_u$ will either be nonnegative or nonpositive in a neighborhood of $p$ while $H(p)$ is nonzero, so by choosing the sign of $\epsilon$ appropriately, the double root $p$ of $G \cdot l_u$ will again induce distinct real roots of $G_\epsilon$.
Hence in either case the other double root of $G \cdot l_u$ gives at most $1$ complex pair of roots of $G_\epsilon$.
\end{proof}

\begin{remark} \label{rem:typeDeterminesRank}
There are some instances in which the type of the apolar ideal determines the almost real rank.
Some cases of this are listed in \Cref{prop:arrankBasicProp}.
Another example of this occurs in degree 6: if a real binary sextic $F$ has an apolar ideal of type $(4,4)$, then $\arrank(F) = 4$.
To see this, note that if $F_\perp, F^\circ$ were both $4^\text{th}$ powers, then $F_\perp - F^\circ$ has almost real roots.
Moreover, if both $F_\perp$ and $F^\circ$ have two pairs of complex roots, then $F_\perp, F^\circ$ is globally positive, in which case a suitable $\RR$-linear combination of $F_\perp$, $F^\circ$ has at least a pair of real roots.
Thus without loss of generality $F_\perp$ has at most 1 root of multiplicity 3, or 2 double roots, or 1 double root and 1 complex pair of roots, and by the reasoning in the proof of \Cref{thm:maxARrank}, there exists a form in $(F)^\perp_4$ with almost real roots.
\end{remark}

We next characterize when the maximal almost real rank of $d-1$ is achieved, which serves as a converse of \Cref{thm:maxARrank}:

\begin{theorem} \label{thm:maxARrankChar}
Let $d \ge 5$ and $F \in \RR[x,y]_d$.
Then $\arrank(F) = d-1 \iff F_\perp$ is a cube of a linear form $\iff (F)^\perp$ contains a cube of a linear form (but no quadratic forms).
\end{theorem}

\begin{proof}
If $F_\perp$ is a cube of a linear form, then $(F)^\perp$ is of type $(3, d-1)$ and $\arrank(F) \ge d-1$ by \Cref{prop:arrankBasicProp}(2, 5).
Conversely, we show that if $d \ge 5$ and $(F)^\perp$ contains no cubes, then $\arrank(F) \le d-2$, by induction on $d$.

We first rule out small types: let $(d_1, d_2)$ be the type of $(F)^\perp$.
If $d_1 \le 3$, then (with the assumptions of no cubes) $\arrank(F) \le d-2$ by \Cref{prop:arrankBasicProp}.
This is enough to cover the base case $d = 5$, and by \Cref{rem:typeDeterminesRank}, this also covers the case $d = 6$.
Thus we assume for the remainder of the proof that $d_1 \ge 4$.

Now suppose $F$ is a form of degree $d \ge 7$.
Note that either $(D_x(F))^\perp$ or $(D_y(F))^\perp$ does not contain a cube of a linear form: if not, say $l_1^3 \in (D_x(F))^\perp$ and $l_2^3 \in (D_y(F))^\perp$, then $(F)^\perp$ would contain 2 independent quartics $xl_1^3, yl_2^3$, which can only happen if $d_1 \le 3$ (since $d_1 = 4 \implies d_2 = d-2 \ge 5$), which has already been covered.
Without loss of generality we may assume $(D_x(F))^\perp$ does not contain a cube of a linear form.
By induction, there is a form $g \in (D_x(F))^\perp$ of degree $\le d-3$ with almost real roots.
Then $xg \in (F)^\perp$ is of degree $\le d-2$, and since $F^\circ \in (F)^\perp$ has degree $\le d-2$ as well, the reasoning in the proof of \Cref{thm:maxARrank} shows that there exists a form in $(F)^\perp_{d-2}$ with almost real roots.
\end{proof}

The characterization above yields sharp bounds on the Hankel index for the curves studied in this paper:

\begin{corollary}
Let $X = \pi_p(C_d)$ be a projection of a rational normal curve $C_d$ away from a point $p \in \PP^d \setminus C_d^3$.
Then $2 \le \eta(X) \le d-4$.
In particular, if $d = 6$, then $\eta(X) = 2$.
\end{corollary}

\begin{proof}
If $\arrank(F(p)) = d-1$, then $(F(p))^\perp$ contains a cube by \Cref{thm:maxARrankChar}.
By apolarity, this implies that $p \in C_d^3$, a contradiction.
Thus $\arrank(F(p)) \le d-2$, so $\eta(X) \le d - 4$ by \Cref{thm:mainThm}.
\end{proof}

As preparation for determining the typical almost real ranks, it is useful to know explicit forms which attain a given almost real rank.
We thus compute the various ranks of monomials $x^{d-i}y^i \in \RR[x,y]_d$.
When $i = 0$, $x^{d-i}y^i = x^d$ is a power of a linear form, hence has real (and complex) [border] rank 1.
By symmetry, we may therefore assume $1 \le i \le \lfloor \frac{d}{2} \rfloor$.
In general, the apolar ideal is
\[
(x^{d-i}y^i)^\perp = (y^{i+1}, x^{d-i+1}).
\]
From this we see that $\cbrank(x^{d-i}y^i) = i+1$ and $\crank(x^{d-i}y^i) = d-i+1$ (cf. \Cref{rem:ranksViaApolar}).
Since $x^{d-i}y^i$ has all real roots, we also have $\rrank(x^{d-i}y^i) = d$.

\begin{proposition} \label{prop:arrkMonomial}
For $d \ge 1$ and $0 \le i \le \lfloor \frac{d}{2} \rfloor$,
\[
\arrank(x^{d-i}y^i) = 
\begin{cases}
1 & \textup{if } i = 0 \\
2 & \textup{if } i = 1 \\
d-1 & \textup{if } i = 2 \\
d-2 & \textup{otherwise}.
\end{cases}
\]
\end{proposition}

\begin{proof}
The cases $i = 0, 1$ follow from \Cref{prop:arrankBasicProp}; the case $i = 2$ is covered by \Cref{thm:maxARrankChar}.
This includes all cases with $d \le 5$.

It thus suffices to show that if $d \ge 6$ and $3 \le i \le \lfloor \frac{d}{2} \rfloor$, then $\arrank(x^{d-i}y^i) > d-3$.
The cases $d = 6$ (resp. $d = 7$) are covered by \Cref{rem:typeDeterminesRank} (resp. \Cref{prop:arrankBasicProp}).
Now suppose $d \ge 8$.
Every form of degree $d-3$ in $(x^{d-i}y^i)^\perp$ can be expressed as
\[
a_{0}x^{d-3} + \ldots + a_{i-4}x^{d-i+1}y^{i-4} + b_{d-i-4}x^{d-i-4}y^{i+1} + \ldots + b_0y^{d-3}
\]
with $(i-3) + (d-i-3) = d - 6$ coefficients $a_0, \ldots, a_{i-4}, b_{d-i-4}, \ldots, b_0 \in \RR$, where we take no $a_i$'s if $i = 3$ (so that the support of this polynomial has a gap of size 4).
By the Descartes' Rule of Signs, the number of distinct nonzero real roots of this polynomial is at most the number of sign changes betwen adjacent coefficients, hence is $\le d - 7$.
Thus $\arrank(x^{d-i}y^i) > d-3$, and so Theorems \ref{thm:maxARrank} and \ref{thm:maxARrankChar} imply that $\arrank(x^{d-i}y^i) = d-2$.
\end{proof}

In particular, we see that for monomial projections, the almost real rank is essentially independent of $i$ (and depends only on whether $X_i$ is contained in the rational normal surface scroll $S(1, d-3)$), and is much larger than the complex border rank (with a gap of at least $\lceil \frac{d}{2} \rceil - 3$, hence the gap is unbounded as $d \to \infty$).

An amusing corollary of \Cref{prop:arrkMonomial} is the existence, in any degree $\ge 4$, of univariate real polynomials with almost real roots whose supports have a gap of size 3, i.e. the Rule of Signs bound is sharp for these polynomials (athough the existence of such polynomials is not sufficient to prove \Cref{prop:arrkMonomial}).
For more on the sharpness of the Rule of Signs bound, cf. \cite{MR1732666}.
% We are now ready to determine the maximal typical almost real rank:

% \begin{theorem} \label{thm:maxTypicalARRank}
% If $d \ge 5$, then $d-2$ is a typical almost real rank.
% \end{theorem}

% \begin{proof}
% By \Cref{thm:maxARrank} and \Cref{thm:maxARrankChar}, it suffices to show that for each $d \ge 5$, there exists a nonempty open set of degree $d$ forms with almost real rank $> d-3$.
% For $5 \le d \le 7$, we may verify this directly: if $d = 5$, then a general form (which is of type $(3,4)$) has almost real rank $3$; the case $d = 6$ is covered by \Cref{rem:typeDeterminesRank}; and for $d = 7$, there is an nonempty open set of forms $F$ of type $(4, 5)$ for which $F_\perp$ has only complex roots (i.e. is a product of 2 strictly positive quadrics).

% For $d \ge 8$, we first restrict to considering degree $d$ forms of generic type, i.e. forms whose apolar ideals are generated in degrees $(\lfloor \frac{d}{2} \rfloor + 1, \lceil \frac{d}{2} \rceil + 1)$ (note that this is the only type which is typical).
% Among such forms, the forms $F$ for which there is an element in $(F)^\perp_{d-3}$ (note that $\lfloor \frac{d}{2} \rfloor + 1 \le d-3$) with almost real roots is stable under small perturbation.
% Thus to show that $d-2$ is a typical almost real rank, it suffices to find a single degree $d$ form of generic type which has almost real rank $d-2$.
% Now by \Cref{prop:arrkMonomial}, the ``balanced'' monomial $x^{\lceil \frac{d}{2} \rceil}y^{\lfloor \frac{d}{2} \rfloor}$ (which is of generic type) has almost real rank $= d-2$.
% \end{proof}

Finally, we consider the problem of determining which almost real ranks are typical.
Our presentation follows that of \cite{MR3348173}.
Recall that a property $P$ of degree $d$ forms is said to be \emph{typical} if, on identifying the set of degree $d$ forms with $\RR^{d+1}$, there is a nonempty Euclidean open set of degree $d$ forms all of which have property $P$.
We say that an almost real rank $r$ is typical if the property ``has almost real rank $= r$'' is typical.
For $F\in\RR[x,y]_d$, we say that $F$ is a \emph{typical form of almost real rank} $r$ if $F$ lies in an open set of $\RR[x,y]_d$ which consists of forms of almost real rank $r$.

Note that the condition ``$(F)^\perp$ contains a cube'' in \Cref{thm:maxARrankChar} is equivalent to saying that $F$ has a real root of multiplicity $\ge d-2$, which is not a typical property.
It follows that $d-1$ is not a typical almost real rank.
Moreover, \Cref{rem:stratification} implies that any $r < \lfloor\frac{d+2}{2}\rfloor$ cannot be a typical almost real rank.
It turns out that these are the only obstructions for an almost real rank to be typical, as will be shown in \Cref{thm:typr}.
To this end, we first characterize the typical forms of a given almost real rank:

\begin{lemma}\label{lem:typicalForms}
Let $F\in\RR[x,y]_d$ with $(F)^\perp$ of generic type, and set $r = \arrank F$.
Then $F$ is a typical form of almost real rank $r$ if and only if all forms in $(F)^\perp_{r-1}$ have at least two pairs of complex roots (counted with multiplicity).
\end{lemma}

\begin{proof}
Suppose that $F$ is typical of almost real rank $r$, and there exists $g \in (F)^\perp_{r-1}$ such that $g$ has at most one pair of complex roots.
In any $\epsilon$-neighborhood of $g$ there exists a form $g_\epsilon$ such that $g_\epsilon$ has almost real roots. 
For any $\epsilon > 0$ we have $\dim (g_\epsilon)_d = \dim (g)_d=d-r+2$, and as $\epsilon$ approaches $0$, $(g_\epsilon)_d$ approaches $(g)_d$.
Therefore the orthogonal complement of $(g_\epsilon)_d$ also approaches the orthogonal complement of $(g)_d$ as $\epsilon$ goes to $0$.
We conclude that in any neighborhood of $F$ there exist forms of almost real rank at most $r-1$, which is a contradiction. 

Conversely, let $F\in\RR[x,y]_d$ with $(F)^\perp$ of generic type and $\arrank F = r$.
Suppose that all forms in $(F)^\perp_{r-1}$ have at least two pairs of complex roots.
For $\epsilon > 0$ sufficiently small, the $\epsilon$-neighborhood of $F$ contains only forms with apolar ideals of generic type (as having non-generic type is a Zariski-closed condition). 
For such $\epsilon$, fix $F_\epsilon$ in the $\epsilon$-neighborhood of $F$.
Within this neighborhood, the ideal $(F_\epsilon)^\perp$ (i.e. the sequence of graded components of $(F_\epsilon)^\perp$) depends continuously on the coefficients of $F_\epsilon$.
Now both conditions ``all forms in $(F)^\perp_{r-1}$ have at most one pair of complex roots'' and ``there exists a form in $(F)^\perp_r$ with almost real roots'' are stable under sufficiently small perturbation, which shows that $F$ is typical of almost real rank $r$.
\end{proof}

\begin{theorem}\label{thm:typr}
For $d \ge 5$, any $r$ with $\lfloor\frac{d+2}{2}\rfloor \le r \le d-2$ is a typical almost real rank.
\end{theorem}
\begin{proof}
We first show that $d-2$ is always a typical almost real rank.
By \Cref{thm:maxARrank} and \Cref{thm:maxARrankChar}, it suffices to show that for each $d \ge 5$, there exists a nonempty open set of degree $d$ forms with almost real rank $> d-3$.
For $5 \le d \le 7$, we may verify this directly: if $d = 5$, then a general form (which is of type $(3,4)$) has almost real rank $3$; the case $d = 6$ is covered by \Cref{rem:typeDeterminesRank}; and for $d = 7$, there is an nonempty open set of forms $F$ of type $(4, 5)$ for which $F_\perp$ has only complex roots (i.e. is a product of 2 strictly positive quadrics).
For $d \ge 8$, it follows from \Cref{lem:typicalForms} and the proof of \Cref{prop:arrkMonomial} that the ``balanced'' monomial $x^{\lceil \frac{d}{2} \rceil}y^{\lfloor \frac{d}{2} \rfloor}$ (which is of generic type) is a typical form of almost real rank $d-2$.

For the remaining ranks, we induct on the degree $d$.
For the base cases $d = 5, 6$, we have that $d-2 = \lfloor \frac{d+2}{2} \rfloor$ is a typical almost real rank by the above.
For the inductive step, fix the following data:
\begin{enumerate}
    \item a rank $\lceil \frac{d+2}{2} \rceil \le r \le d-2$,
    \item a typical form $F \in \RR[x,y]_d$ of almost real rank $r$ (by perturbing $F$ if necessary, we may assume that $(F)^\perp = (F_\perp, F^\circ)$ is of generic type),
    \item a nonzero form $S := C_1 F_\perp + C_2 F^\circ \in (F)^\perp_r$ with almost real roots.
\end{enumerate}
We will exhibit a form $H$ of degree $d+1$ such that $(H)^\perp$ is of generic type, $(H)^\perp \subseteq (F)^\perp$, and $S \in (H)^\perp$.
By \Cref{lem:typicalForms} this shows that $H$ is a typical form of almost real rank $r$, so $r$ is a typical almost real rank in degree $d+1$.
This is enough for the induction, since we already know that $(d+1)-2$ is a typical almost real rank in degree $d+1$ (note also that $\lceil \frac{d+2}{2} \rceil = \lfloor \frac{(d+1)+2}{2} \rfloor$).
We consider two cases depending on the parity of $d$, namely $d = 2k$ for $k \ge 3$, or $d = 2k-1$ for $k \ge 4$.

First, suppose $d = 2k-1$ is odd, so that $\deg F_\perp = k$, $\deg F^\circ = k+1$.
We claim that there exists a linear form $L \in \RR[x,y]_1$ such that $C_1 - LC_2$ has a real root which is not a root of $LF_\perp + F^\circ$.
If not, then for every linear form $L$, we have that every root of $C_1 - LC_2$ is a root of $LF_\perp + F^\circ$.
Now for any $(a,b) \in \RR^2$ with $F_\perp(a,b) \ne 0$ and $C_2(a,b) \ne 0$, there exists a linear form $L$ such that $L(a,b) = \frac{C_1(a,b)}{C_2(a,b)}$, i.e. $(a,b)$ is a root of $C_1 - LC_2$.
By assumption $(a,b)$ is also a root of $LF_\perp + F^\circ$, so $L(a,b) = \frac{-F^\circ(a,b)}{F_\perp(a,b)}$.
Varying over such $(a,b)$, we see that the two rational functions $C_1/C_2$ and $-F^\circ/F_\perp$ agree at infinitely many points, hence must be equal.
But this implies that $S = C_1F_\perp + C_2F^\circ = 0$, a contradiction.
We conclude that such an $L$ exists. 
For such $L$, set $G := L F_\perp + F^\circ$, write $C_1 - LC_2 = L_1K$, where $L_1 \in \RR[x,y]_1$ does not divide $G$, and take $H$ to be the unique form of degree $d+1$ with apolar ideal generated by $(L_1F_\perp, G)$.
Then $(H)^\perp \subseteq (F)^\perp$, and $S = (C_1-L C_2)F_\perp + C_2 G = K(L_1F_\perp) + C_2 G \in (H)^\perp$ as desired.

The reasoning in the case $d = 2k$ is similar: here $\deg(F_\perp) = \deg(F^\circ) = k+1$.
We claim that there exists $\alpha \in \RR$ such that $C_1 - \alpha C_2$ has a real root which is not a root of $\alpha F_\perp + F^\circ$.
This follows from the same reasoning as in the case $d = 2k-1$ (in fact even simpler, since there is no choice involved in the scalar $\alpha$, as opposed to a linear form).
Having obtained such an $\alpha$, we set $G := \alpha F_\perp + F^\circ$, write $C_1 - \alpha C_2 = L_0 K$, where $L_0 \in \RR[x,y]_1$ does not divide $G$, and take $H$ to be the unique form of degree $d+1$ with apolar ideal generated by $(L_0F_\perp, G)$.
Then as before, $(H)^\perp \subseteq (F)^\perp$ and $S = (C_1 - \alpha C_2)F_\perp + C_2G = K(L_0 F_\perp) + C_2G \in (H)^\perp$.
\end{proof}

% \begin{example}
% In general, the set of degree $d$ forms apolar to a given form $g$ of degree $r \le d$ is a vector space of dimension $r$ (e.g. by \Cref{lem:genApolarity}).
% (as can be seen by viewing $\langle g, \cdot \rangle$ as a linear map $k[x,y]_d \to k[x,y]_{d-d_1}$).
% Taking $d = 5$ and $g = x^2 + y^2$ gives a $2$-dimensional space of binary quintics which are apolar to $g$, namely $\Span\{ x^5 - 10x^3y^2 + 5xy^4, y^5 - 10x^2y^3 + 5x^4y \}$.
% Now
% \begin{align*}
%     x^5 - 10x^3y^2 + 5xy^4 &= x(x^4 - 10x^2y^2 + 25y^4 - 20y^4) \\
%     &= x(x^2 - (5 + \sqrt{20})y^2)(x^2 - (5 - \sqrt{20})y^2)
% \end{align*}
% has real roots $0, \pm \sqrt{5 \pm 2\sqrt{5}}$, which shows that the condition ``$\arrank(F) \ge 3$'' in \Cref{thm:maxTypicalARRank} cannot be dropped.
% \end{example}

\bibliographystyle{acm}
\bibliography{Hankel_index_almost_real_rank.bbl}

\end{document}